\documentclass{article}

\usepackage[title]{appendix}
\usepackage[round]{natbib}

\usepackage[margin=1in]{geometry}
\usepackage{amsmath}
\usepackage{amssymb}
\usepackage{amsthm}
\usepackage{booktabs}  
\usepackage{color}
\usepackage{graphicx}
\usepackage{hyperref}
\usepackage{bbm}
\usepackage{float}
\usepackage{mathtools}
\usepackage{enumitem}
\usepackage{afterpage}

\usepackage{subfig}
\usepackage{mathtools}

\DeclarePairedDelimiterX{\infdivx}[2]{(}{)}{%
  #1\;\delimsize\|\;#2%
}
\newcommand{\infdiv}{D\infdivx}

\usepackage[ruled,vlined]{algorithm2e}

\SetKwBlock{Repeat}{repeat}{end}

\newlength\mylen
\newcommand\myinput[1]{%
  \settowidth\mylen{\KwIn{}}%
  \setlength\hangindent{\mylen}%
  \hspace*{\mylen}#1\\}

\renewcommand{\d}{{\bf d}}

\newcommand{\p}{{\bf p}}

\renewcommand{\u}{{\bf u}}

\newcommand{\x}{{\bf x}}
\newcommand{\X}{{\bf X}}

\newcommand{\e}{{\bf e}}
\newcommand{\y}{{\bf y}}

\newcommand{\z}{{\bf z}}

\def\Xm{{\bf X}}

\newcommand{\D}{{\mathcal D}}
\newcommand{\G}{{\mathcal G}}

\newcommand{\F}{{\mathcal F}}

\newcommand{\R}{{\mathbb{R}}}

\newtheorem{theorem}{Theorem}
\newtheorem{lemma}[theorem]{Lemma}

\newtheorem{assumption}{Assumption}
\newtheorem{definition}{Definition}

\newtheorem*{theorem*}{Theorem}
\newtheorem*{lemma*}{Lemma}
\newtheorem*{corollary*}{Corollary}

\DeclarePairedDelimiter\ceil{\lceil}{\rceil}

\newcommand{\cm}{{\kappa}} 

\DeclareMathOperator*{\argmax}{arg\,max}
\DeclareMathOperator*{\argmin}{arg\,min}

\usepackage[breakable]{tcolorbox}
\newtcolorbox{mybox}[2][]
{
  breakable,
  colframe = #2!25,
  colback  = #2!25!white!25,
  left = 0.1mm,
  right=0.1mm,
  #1
}

\usepackage{url}
\usepackage{multirow}
\usepackage{array}

\title{Direct-Search for a Class of Stochastic Min-Max Problems}

\author{
    \normalsize
    Sotiris Anagnostidis \\
    \normalsize
    Institute for Machine Learning\\
    \normalsize
    ETH Zürich\\
    \texttt{\small sotirios.anagnostidis@inf.ethz.ch}
  \and
  \normalsize
    Aurelien Lucchi \\
    \normalsize
    Institute for Machine Learning\\
    \normalsize
    ETH Zürich\\
    \texttt{\small aurelien.lucchi@inf.ethz.ch}
    \and
    \normalsize
    Youssef Diouane \\
    \normalsize
    ISAE-SUPAERO\\
    \normalsize
    Université de Toulouse, France\\
    \texttt{\small youssef.diouane@isae-supaero.fr}
}
\date{}

\begin{document}
\maketitle

\begin{abstract}
Recent applications in machine learning have renewed the interest of the community in min-max optimization problems. While gradient-based optimization methods are widely used to solve such problems, there are however many scenarios where these techniques are not well-suited, or even not applicable when the gradient is not accessible. We investigate the use of direct-search methods that belong to a class of derivative-free techniques that only access the objective function through an oracle. In this work, we design a novel algorithm in the context of min-max saddle point games where one sequentially updates the min and the max player. We prove convergence of this algorithm under mild assumptions, where the objective of the max-player satisfies the Polyak-\L{}ojasiewicz (PL) condition, while the min-player is characterized by a nonconvex objective. Our method only assumes dynamically adjusted accurate estimates of the oracle with a fixed probability. To the best of our knowledge, our analysis is the first one to address the convergence of a direct-search method for min-max objectives in a stochastic setting.
\end{abstract}


\section{Introduction}
\label{sec:introduction}

Recent applications in the field of machine learning, including generative models~\citep{goodfellow2014generative} or robust optimization~\citep{ben2009robust}, have triggered significant interest for the optimization of min-max functions of the form
\begin{align}
\min_{\x \in \mathcal{X}} \max_{\y \in \mathcal{Y}} f(\x, \y) = \mathbb{E}[\tilde{f}(\x, \y, \xi)],
\label{eq:min-max}    
\end{align}
where $\xi$ is a random variable characterized by some  distribution. In machine learning, $\xi$ is for instance often drawn from a distribution that depends on the training data.

In practice, min-max problems are often solved using gradient-based algorithms, especially simultaneous gradient descent ascent (GDA) that simply alternates between a gradient descent step for $\x$ and a gradient ascent step for $\y$. While these algorithms are attractive due to their simplicity, there are however cases where the gradient of the objective function is not accessible, such as when modelling distributions with categorical variables~\citep{jang2016categorical}, tuning hyper-parameters~\citep{audet2006finding, marzat2011min} and multi-agent reinforcement learning with bandit feedback~\citep{zhang2019multi}. A resurgence of interest has recently emerged for applications in black-box optimization~\citep{bogunovic2018adversarially, liu2019min} and black-box poisoning attack~\citep{liu2020min}, where an attacker deliberately modifies the training data in order to tamper with the model's predictions. This can be formulated as a min-max optimization problem, where only stochastic accesses to the objective function are available~\citep{wang2020zeroth}.

In this work, we investigate the use of direct-search methods to optimize min-max objective functions without requiring access to the gradients of the objective $f$. Direct-search methods have a long history in the field of optimization, dating back from the seminal paper of~\cite{hooke1961direct}. The appeal of these methods is due to their simplicity but also potential ability to deal with non-trivial objective functions. Although there are variations among random search techniques, most of them can be summarized conceptually as sampling random directions from a search space and moving towards directions that decrease the objective function value. We note that these techniques are sometimes named derivative-free methods, but it is important to distinguish them from other techniques that try to estimate derivatives based on finite difference~\citep{spall2003stochastic} or smoothing~\citep{nesterov2017random}. We refer the reader to the surveys by~\cite{lewis2000direct, rios2013derivative} for a comprehensive review of direct-search methods.

Solving the saddle point problem~\eqref{eq:min-max} is equivalent to finding a saddle point\footnote{In the game theory literature, such point is commonly
referred to as (global) Nash equilibrium, see e.g.~\cite{liang2018interaction}.}  $(\x^*, \y^*)$ such that
\begin{equation*}
f(\x^*, \y) \leq f(\x^*, \y^*) \leq f(\x, \y^*) \quad \forall \x \in \mathcal{X}, \quad \forall \y \in \mathcal{Y}.
\end{equation*}
There is a rich literature on saddle point optimization
for the particular class of convex-concave functions (i.e.
when $f$ is convex in $\x$ and concave in $\y$) that are differentiable. Although this type of objective function is commonly encountered in applications such as constrained convex minimization, many saddle point problems of interest do not satisfy the convex-concave assumption. This for instance includes applications such as Generative Adversarial Networks (GANs)~\citep{goodfellow2014generative}, robust optimization~\citep{ben2009robust, bogunovic2018adversarially} and multi-agent reinforcement learning~\citep{omidshafiei2017deep}. For min-max problems without access to derivatives, the literature is in fact very scarce. Most existing techniques such as~\cite{hare2013derivative, hare2013bderivative, custodio2019worst} consider finite-max functions, i.e., functions of the form $f(\x) = \max \{ f_i(\x) : i = 1, 2, \ldots , N \}$ where $N>0$ is finite and each $f_i$ is continuously differentiable. Other techniques such as~\cite{bertsimas2010robust, bertsimas2010brobust} are restricted to functions $f$ that are convex with respect to $\x$ or only provide asymptotic convergence analysis~\citep{menickelly2020derivative}. We refer the reader to Section~\ref{sec:related_work} for a more detailed discussion of prior approaches.

Motivated by a wide range of applications, we therefore focus on a nonconvex and nonconcave stochastic setting where the $\max$ player satisfies the PL condition (see Definition~\ref{def:PL-condition} in Section~\ref{sec:convergence-min}), which is known to be a weaker assumption compared to convexity~\citep{karimi2015linear}. In summary, our main contributions are:
\begin{itemize}
    \item We design a novel direct-search algorithm for such $\min$-$\max$ problems and provide non-asymptotic convergence guarantees in terms of first-order Nash equilibrium. Concretely, we prove convergence to an $\epsilon$-first-order Nash Equilibrium (for a definition see Section~\ref{sec:preliminaries}) in $\mathcal{O}(\epsilon^{-2}\log(\epsilon^{-1}))$ iterations, which is comparable to the rate achieved by gradient-based techniques~\citep{nouiehed2019solving}.
    \item We derive theoretical convergence guarantees in a stochastic setting where one only has access to accurate estimates of the objective function, with some fixed probability.
    We prove our results for the case where the min player optimizes a nonconvex function while the max player optimizes a PL function.
    \item We validate empirically our theoretical findings, including settings where derivatives are not available.
\end{itemize}


\section{Related Work}
\label{sec:related_work}

\paragraph{Direct-search methods for minimization problems}
The general principle behind direct-search methods is to optimize a function $f(\x)$ without having access to its gradient $\nabla f(\x)$. There is a large number of algorithms that are part of this broad family including golden-section search techniques or random search~\citep{rastrigin1963convergence}. Among the most popular algorithms in machine learning are evolution strategies and population-based algorithms that have demonstrated promising results in reinforcement learning~\citep{salimans2017evolution, maheswaranathan2018guided} and bandit optimization~\citep{flaxman2004online}. At a high-level, these techniques work by maintaining a distribution over parameters and duplicate the individuals in the population with higher fitness. Often these algorithms are initialized at a random point and then adapt their search space, depending on which area contains the best samples (i.e. the lowest function value when minimizing $f(\x)$). New samples are then generated from the best regions in a process repeated until convergence. The most well-known algorithms that belong to this class are evolutionary-like algorithms, including for instance CMA-ES~\citep{hansen2003reducing}. Evolutionary strategies have recently been shown to be able to solve various complex tasks in reinforcement learning such as Atari games or robotic control problems, see e.g.~\cite{salimans2017evolution}. Their advantages in the context of reinforcement learning are their reduced sensitivity to noisy or uninformative gradients (potentially increasing their ability to avoid local minima~\citep{conti2017improving}) and the ease with which one can implement a distributed or parallel version. 

\paragraph{Convergence guarantees for direct-search methods}
Proofs of convergence for direct-search methods are based on a specific construction for the sampling directions, often that they positively span the whole search space~\citep{conn2009introduction}, or that they are dense in certain types of directions (known as refining directions) at the limit point~\citep{audet2006mesh}. In addition, they also typically rely on the use of a forcing function that imposes each new selected iterate to decrease the function value adequately. This technique has been analyzed in~\cite{vicente2013worst} who proved convergence under mild assumptions in $\mathcal{O}(\epsilon^{-2})$ iterations for the goal $\| \nabla f(\x) \| < \epsilon$. The number of required steps is reduced to $\mathcal{O}(\epsilon^{-1})$ for convex functions $f$, and to $\mathcal{O}(\log(\epsilon^{-1}))$ for strongly-convex functions~\citep{konevcny2014simple}. This is on par with the steepest descent method for unconstrained optimization \citep{nesterov2013introductory} apart from some constants that depend on the dimensionality of the problem.

\paragraph{Stochastic estimates of the function}
In our analysis, we only assume access to stochastic estimates of the objective function.
\begin{equation}
    f(\x, \y) = \mathbb{E}[\tilde{f}(\x, \y, \xi)],
    \label{eq:stochastic_fn}
\end{equation}
where $\xi$ is a random variable that captures the randomness of the objective function. The origin of the noise could be privacy related, or caused by a noisy adversary. Most commonly, it might arise from online streaming data, distributed and batch-sized updates due to the sheer size of the problem. Stochastic gradient descent is often used to optimize Eq.~\eqref{eq:stochastic_fn}, where one often assumes access to accurate estimates of $f$ and consider updates only in expectation~\citep{johnson2013accelerating}. To establish similar convergence rates to the deterministic case, an alternative solution consists of adapting the accuracy of these estimates dynamically, which can be ensured by averaging multiple samples together. This approach has for instance been analyzed in the context of trust-region methods~\citep{blanchet2019convergence} and line-search methods~\citep{paquette2018stochastic,bergou2019stochastic}, including direct-search for the minimization of nonconvex functions~\citep{dzahini2020expected}.

\paragraph{Algorithms for finding equilibria in games}
Since the pioneering work of von~\cite{neumann1928theorie}, equilibria in games have received great attention. Most past results focus on convex-concave settings~\citep{chen2014optimal, hien2017inexact}. Notably,~\cite{cherukuri2017saddle} studied convergence of the GDA algorithm under strictly convex-concave assumptions. For problems where the function does not satisfy this condition however, convergence to a saddle point is not guaranteed. More recent results focus on relaxing these conditions. The work of~\cite{nouiehed2019solving} analyzed gradient descent-ascent under a similar scenario, where the objective of the max player satisfies the PL condition and where the min player optimizes a nonconvex objective.~\cite{ostrovskii2020efficient, wang2020zeroth} analyze a nonconvex-concave class of problems, while~\cite{lin2020gradient} present a two-scale variant of the GDA algorithm for a similar scenario, providing a replacement for the alternating updates scheme.

We take inspiration from the work of~\cite{liu2019min, nouiehed2019solving, sanjabi2018convergence} to design a novel alternating direct-search algorithm, where the inner maximization problem is solved almost exactly before performing a single step towards improving the strategy of the minimization player. We are able to prove convergence of our direct-search algorithm under this procedure, which has been proven to be more stable than the analogous simultaneous one, as rigorously shown in~\cite{gidel2018negative} and \cite{zhang2019convergence} for a variety of algorithms.


\section{Preliminaries} \label{sec:preliminaries}

Throughout, we use $\|.\|$ to denote the Euclidean norm; that is, for $\x \in \R^n$ we have $\| \x \| = \sqrt{\x^\intercal \x}$.

\subsection{Min-Max Games}

We consider the optimization problem defined in Eq.~\eqref{eq:min-max} for which a common notion of optimality is the concept of Nash equilibrium as mentioned previously, which is formally defined as follows.

\begin{definition}
We say that a point $(\x^*, \y^*) \in \mathcal{X} \times \mathcal{Y}$ is a Nash equilibrium of the game if
\begin{equation*}
    f(\x^*, \y) \leq f(\x^*, \y^*) \leq f(\x, \y^*) \quad \forall \x \in \mathcal{X}, \quad \forall \y \in \mathcal{Y}.
\end{equation*}
\end{definition}
A Nash equilibrium is a point where the change of strategy of each player individually does not lead to an improvement from her viewpoint. Such a Nash equilibrium point always exists for convex-concave games~\citep{jin2019minmax}, but not necessarily for nonconvex-nonconcave games. Even when they exist, finding Nash equilibria is known to be a NP-hard problem, which has led to the introduction of local characterizations as discussed in~\cite{jin2019minmax, adolphs2018local}. Here we use the notion of a first-order Nash equilibrium (FNE) (for a definition we refer to~\cite{pang2016unified}). We focus on the problem of converging to such a FNE point, or an approximate FNE defined as follows (adapted from \cite{nouiehed2019solving} in the absence of constraints).

\begin{definition}
For a function $f: \R^n \times \R^m \to \R$, a point $(\x^*, \y^*) \in \R^n \times \R^m$ is said to be an $\epsilon$-first-order Nash Equilibrium ($\epsilon$-FNE) if:
$\| \nabla_{\x} f(\x^*, \y^*) \| \leq \epsilon$ and $\| \nabla_{\y} f(\x^*, \y^*) \| \leq  \epsilon$.
\end{definition}

\subsection{Direct-Search Methods}

\paragraph{Spanning set}
Direct-search methods typically rely on the smoothness of the objective function, which we denote by $f: \R^n \to \R$ in this section\footnote{When using the function $f$ with one set of variables, we consider the minimization problem. When using two sets of variables, we instead consider the min-max problem as defined in Eq.~\eqref{eq:min-max}.}, and on appropriate choice of sampling points to prove convergence. The key idea to guarantee convergence is that one of the sampled directions will form an acute angle with the negative gradient. This can be ensured by sampling from a Positive Spanning Set (PSS). The quality of a spanning set $\D$ is typically measured using a notion of cosine measure defined as

\begin{equation}
\cm(\D) = \min_{\textbf{0} \neq \u \in \R^n} \max_{\d \in \D} \frac{\u^T \d}{\| \u \| \| \d \|}.
\end{equation}

In the following, we will consider positive spanning sets such that $\cm(\D) \geq \kappa_{\min} > 0$ and $d_{\min} \leq \| \d \| \leq d_{\max}, \, \forall \d \in \D$. These assumptions require $|\D| \geq n + 1$. Common choices are i) the positive and negative orthonormal bases $\D = [I_n -I_n] = [\e_1, \dots, \e_n, -\e_1, \dots, -\e_n]$ of size $|\D| = 2n$, ii) a minimal positive basis with uniform angles of size $|\D| = n + 1$ (see Corollary 2.6 of~\cite{conn2009introduction} and~\cite{kolda2003optimization}) or iii) even rotations of these matrices~\citep{gratton2016second}.

\paragraph{Forcing function}
Another critical component to guarantee that the function value decreases at each step appropriately is a forcing function $\rho$ that satisfies $\frac{\rho(\sigma)}{\sigma} \to 0$ when $\sigma \to 0$. Given such $\sigma$, direct-search methods sample new points according to the rule
\begin{equation}
\x^{\prime} = \x + \sigma \d,
\end{equation}
and accept points for which
\begin{equation}
     f(\x^{\prime}) < f(\x) - \rho(\sigma), \label{eq:sufficient-decrease}
\end{equation}
$\d \in \D$. If the previous condition holds for some $\d \in \D$, then the new point is accepted, the step is deemed successful and the $\sigma$ parameter is increased, otherwise $\sigma$ is decreased and the above process is repeated. We use a parameter $\gamma$ to indicate these updates of the step size. For convenience and without loss of generality, we will only consider spanning sets with vectors of unitary length $d_{\min} = d_{\max} = 1$ and a forcing function 
\begin{equation*}
    \rho(\sigma) = c \sigma^2.
\end{equation*} 
The direct-search scheme is displayed in Algorithm~\ref{algo:direct-search}.

\begin{algorithm}
\SetAlgoLined
\caption{{\sc Direct-search($f, \x_0, c, T$)}}
\label{algo:direct-search}
\KwIn{$f$: objective function, with $f_k$ it's estimate at step $k$}
\myinput{$c$: forcing function constant}
\myinput{$T$: number of steps}
Initialize step size value $\sigma_0$.
Choose $\gamma > 1$. Create the Positive Spanning Set $\D$.

\For{k = 0, \dots, T - 1}{
\textbf{1. Offspring generation:}\\
Generate the points $ \x^i = \x_{k} + \sigma_{k} \d^i, \quad \forall \d^i \in \D.$ \\
\textbf{2. Parent Selection:} \\
Choose $\x^{\prime} = \argmin_i f_k(\x^i)$.\\

\textbf{3. Sufficient Decrease:}

\uIf{ $f_k(\x^{\prime}) < f_k(\x_k) - \rho(\sigma_k)$ }{
    (Iteration is successful)\\
    Update and increase step size
    $\x_{k + 1} = \x^{\prime}, \sigma_{k + 1} = \min\{\sigma_{\max}, \gamma \sigma_k\}$.
}\Else{
    (Iteration is unsuccessful)\\
    Decrease step size 
    $\x_{k + 1} = \x_k$, $\sigma_{k + 1} = \gamma^{-1} \sigma_k$.
}}
\Return  $\x_{T}$
\end{algorithm}

\section{Stochastic Direct-Search} \label{sec:convergence-min}
The full algorithm we analyze to solve the min-max objective is presented in Algorithm~\ref{algo:min-max-DR}. It consists of two steps: i) first solve the maximization problem w.r.t. the $\y$ variable using Algorithm~\ref{algo:direct-search}, and ii) perform one update step for the $\x$ variable. In this section, we first analyze the convergence properties of Algorithm~\ref{algo:direct-search} in the setting where we only have access to estimates of the objective function $f$,
\begin{equation*}
    f(\x) = \mathbb{E}[\tilde{f}(\x, \xi)].
\end{equation*}
Let $(\Omega, \F, P)$ be a probability space with elementary events denoted with $\omega$. We denote the random quantities for the iterate by $\x_k = \X_k(\omega)$ and for the step size by $\sigma_k = \Sigma_k(\omega)$. Similarly let $\{F_k^0, F_k^{\sigma}\}$ be the estimates of $f(\X_k)$ and $f(\X_k + \Sigma_k \d_k)$, for each $\d_k$ in a set $\D$, with their realizations $f_k^0 = F_k^0(\omega)$, $f_k^{\sigma} = F_k^{\sigma}(\omega)$. At each iteration the influence of the noise on function evaluations is random. We will assume that, when conditioned on all the past iterates, these estimates are sufficiently accurate with a sufficiently high probability. We formalize this concept in the two definitions below.

\begin{definition} ($\epsilon_f$-accurate) \label{def:accurate}
The estimates $\{F_k^0, F_k^{\sigma}\}$ are said to be $\epsilon_f$-accurate with respect to the corresponding sequence if
\begin{equation*}
    | F_k^0 - f(\X_k) | \leq \epsilon_f \Sigma_k^2 \, \text{ and } \, | F_k^{\sigma} - f(\X_k + \Sigma_k \d_k) | \leq \epsilon_f \Sigma_k^2.
\end{equation*}
\end{definition}

\begin{definition} ($p_f$-probabilistically $\epsilon_f$-accurate) \label{def:probabilistically-accurate}
The estimates $\{F_k^0, F_k^{\sigma}\}$ are said to be $p_f$-probabilistically $\epsilon_f$-accurate with respect to the corresponding sequence if the events
\begin{equation*}
  J_k = \{\text{The estimates $\{F_k^0, F_k^{\sigma}\}$ are $\epsilon_f$-accurate}\}  
\end{equation*}
satisfy the condition\footnote{We use $1_A$ to denote the indicator function for the set $A$ and $A^c$ to denote its complement.}
\begin{equation*}
    P(J_k \mid \F_{k-1}) = \mathbb{E}[{1_{J_k} \mid \F_{k-1}}] \geq p_f,
\end{equation*}
where $\F_{k-1}$ is the sigma-algebra generated by the sequence $\{ F_0^0, F_0^{\sigma},\dots, F_{k-1}^0, F_{k-1}^{\sigma}\}$.

\end{definition}

As the step size $\sigma$ gets smaller, meaning that we are getting closer to the optimum, we require the accuracy over the function values to increase. However, the probability to encounter a good estimation remains the same throughout. A significant challenge arises, as steps may satisfy our sufficient decrease condition specified in Eq.~\eqref{eq:sufficient-decrease} falsely, leading to a potential increase in terms of the objective value. This increase can potentially be very large, leading to divergence, and we therefore need to require an additional assumption regarding the variance of the error.

\begin{assumption} \label{ass:variance-condition}
The sequence of estimates $\{F_k^0, F_k^{\sigma}\}$ are said to satisfy a $l_f$-variance condition if for all $k\geq 0$
\begin{align*}
    \mathbb{E}[|F_k^0 - f(\X_k)|^2 \mid \F_{k-1}] &\leq l_f^2 \Sigma_k^4, \\
    \mathbb{E}[|F_k^{\sigma} - f(\X_k + \Sigma_k \d_k)|^2 \mid \F_{k-1}] &\leq l_f^2 \Sigma_k^4.
\end{align*}
\end{assumption}

Based on the above assumptions, we reach the following conclusion regarding inaccurate steps (similar to Lemma 2.5 in~\cite{paquette2018stochastic}).

\begin{mybox}{gray}
\begin{lemma} \label{lemma:bound-absolute}
Let Assumption~\ref{ass:variance-condition} hold for $p_f$-probabilistically $\epsilon_f$-accurate estimates of a function. Then for $k \geq 0 $ we have
\begin{align*}
    &\mathbb{E}[1_{J_k^c}|F_k^0 - f(\X_k)| \mid \F_{k-1}] \leq (1 - p_f)^{1/2} l_f \Sigma_k^2, \\
    &\mathbb{E}[1_{J_k^c}|F_k^{\sigma} - f(\X_k + \Sigma_k \d_k)| \mid \F_{k-1}] \\ &\qquad\qquad \leq (1 - p_f)^{1/2} l_f \Sigma_k^2.
\end{align*}
\end{lemma}
\end{mybox}

\paragraph{Computing the estimates}
In order to satisfy Assumption~\ref{ass:variance-condition} we can perform multiple function evaluations and average them out (see for instance~\cite{tropp2015introduction}). We therefore get an estimate $F_k^0 = \frac{1}{|S_k^0|}\sum_{\xi_i \in S_k^0} \tilde{f}(\X_k, \xi_i)$, where $S_k^0, S_k^{\sigma}$ correspond to independent samples for $F_k^0$ and $F_k^{\sigma}$ respectively. Assuming bounded variance, i.e. $\mathbb{E}[|\tilde{f}(\x, \xi) - f(\x)|^2] \leq \sigma_f^2$, known concentration results (see e.g.~\cite{tripuraneni2018stochastic, chen2018stochastic}) guarantee that we can obtain $p_f$-probabilistically $\epsilon_f$-accurate estimates for 
\begin{equation*}
    |S_k^0| \geq \mathcal{O}(1) \left( \frac{\sigma_f^2}{\epsilon_f^2 \Sigma_k^4}\log \left(\frac{1}{1 - p_f} \right) \right)
\end{equation*}
number of evaluations (the same result holds for $S_k^{\sigma}$). To also satisfy Assumption~\ref{ass:variance-condition}, we additionally require $|S_k^0| \geq \frac{\sigma_f^2}{l_f \Sigma_k^4}$.

\subsection{Convergence of Stochastic Direct-Search}

In order to study the convergence properties of Algorithm~\ref{algo:direct-search}, we introduce the following (random) Lyapunov function:
\begin{equation*}
    \Phi_k = v (f(\X_k) - f^*) + (1 - v) \Sigma_k^2,
\end{equation*}
where $v \in (0,1)$ is a constant. We denote by $f^*$ the minimum of the function $f$, assumed to exist and potentially achieved at multiple positions. The Lyapunov function $\Phi_k$ will be used to track the progress of the gradient norm $\| \nabla f(\X_k) \|$, which will serve as a measure of convergence. 

Theorem~\ref{theorem:pfi-k-nonconvex} presented below ensures that the Lyapunov function decreases over iterations. Using this result, one can guarantee that the sequence of step-sizes decreases and then exploit the fact that for sufficiently small step sizes (and accurate estimates), the steps are successful, i.e. they decrease the objective function. The proof of the next Theorem is mainly inspired by~\cite{dzahini2020expected, audet2019stomads}.

\begin{mybox}{gray}
\begin{theorem} \label{theorem:pfi-k-nonconvex}
Let a function $f$ with a minimum value $f^*$, with Lipschitz continuous gradients with a constant $L$. Let also $f$ be $p_f$-probabilistically $\epsilon_f$-accurate, while also having bounded noise variance according to Assumption~\ref{ass:variance-condition} with constant $l_f$. Then:
\begin{equation}
    \mathbb{E}[\Phi_{k + 1} - \Phi_k \mid \F_{k-1}] \leq - p_f (1 - v) (1 - \frac{1}{\gamma^2}) \frac{\Sigma_k^2}{2}.
\end{equation}
The constants $c$, $v$ and $p_f$ should satisfy
\begin{equation*}
    c - 2 \epsilon_f > 0 \text{, } \quad \frac{p_f}{\sqrt{1 - p_f}} \geq \frac{4v l_f}{(1 - v) (1 - \gamma^{-2})} \text{,}
\end{equation*}
\begin{equation*}
    \frac{v}{1 - v} \geq \frac{1}{c - 2 \epsilon_f}(\gamma^2 - \frac{1}{\gamma^2}).
\end{equation*}
\end{theorem}
\end{mybox}

Next, we characterize the number of steps required to converge by using a renewal-reward process adapted from~\cite{blanchet2019convergence}. Let us define the random process $\{ \Phi_k, \Sigma_k\}$, with $\Phi_k \geq 0$ and $\Sigma_k \geq 0$. Let us also denote with $W_k$ a random walk process and $\F_k$ the $\sigma$-algebra generated by $\{ \Phi_0, \Sigma_0, W_0, \dots, \Phi_k, \Sigma_k, W_k\}$ with $W_0$ = 1, 
\begin{equation} \label{eq:wk-updates}
\begin{split}
    P(W_{k + 1} &= 1 \mid \F_k) = p, \\ P(W_{k + 1} &= -1 \mid \F_k) = 1 - p.
\end{split}
\end{equation}
We also define a family of stopping times $\{T_{\epsilon}\}_{\epsilon > 0}$ with respect to $\{ \F_k\}_{k \geq 0}$ for $\epsilon > 0$. 

\begin{assumption} \label{ass:iterates}
Given the random quantities $\{ \Phi_k, \Sigma_k, W_k\}$, we make the following assumptions.
\renewcommand{\theenumi}{\roman{enumi}}
\begin{enumerate}
    \item There exists $\lambda > 0$ such that $\Sigma_{\max} = \Sigma_0 e^{\lambda j_{\max}}$ for $j_{\max} \in \mathbb{Z}$, and $\Sigma_k \leq \Sigma_{\max}$ for all $k$.
    \item There exists $\Sigma_{\epsilon} = \Sigma_0 e^{\lambda j_{\epsilon}}$ with $j_{\epsilon} \in \mathbb{Z}$, such that 
    \begin{equation*}
        1_{T_{\epsilon} > k} \Sigma_{k + 1} \geq 1_{T_{\epsilon} > k} \min\{ \Sigma_k e^{\lambda W_{k + 1}}, \Sigma_{\epsilon}\}
    \end{equation*}
    where $W_{k + 1}$ satisfies Equation~\eqref{eq:wk-updates} with probability $p > \frac{1}{2}$.
    \item There exists a nondecreasing function $h(\cdot)\!\!\!\!\!\!: [0, \infty] \to (0, \infty)$ and a constant $\Theta > 0$ such that
    \begin{equation*}
        1_{T_{\epsilon} > k} \mathbb{E}[\Phi_{k + 1} \mid \F_k] \leq 1_{T_{\epsilon} > k} (\Phi_k -\Theta h(\Sigma_k)).
    \end{equation*}
\end{enumerate}
\end{assumption}

Assumption~\ref{ass:iterates} (ii) requires that step sizes tend to increase when below a specific threshold, while Assumption~\ref{ass:iterates} (iii) requires that the random function $\Phi$ decreases in expectation (already proved in Theorem~\ref{theorem:pfi-k-nonconvex}). Under this assumption, the following results hold for the stopping time $T_{\epsilon}$~\citep{blanchet2019convergence}.

\begin{mybox}{gray}
\begin{theorem} \label{theorem:stopping-time}
Under Assumption~\ref{ass:iterates}, we have
\begin{equation*}
    \mathbb{E}[T_{\epsilon}] \leq \frac{p}{2p - 1} \frac{\Phi_0}{\Theta h(\Sigma_{\epsilon})} + 1.
\end{equation*}
\end{theorem}
\end{mybox}

In our case, we use the fundamental result of convergence for direct-search methods, that comes from correlating the norm of the gradient with the step size for unsuccessful iterations (generalization of results in~\cite{vicente2013worst, gratton2016second}).

\begin{mybox}{gray}
\begin{lemma} \label{lemma:unsuccessful-step}
Let $f: \x \in \R^n \to \R$ be a continuous differentiable function with Lipschitz continuous gradients of constant $L$. Let also $\D$ be a positive spanning set with $\cm(\D) = \kappa_{\min} > 0$ and vectors $\d$ satisfying $\| \d \| = 1,\, \forall \d \in \D$. For a forcing function $\rho(\sigma) = c \sigma^2$ and an $\epsilon_f$-accurate estimates of the function, for an unsuccessful step $k$ it holds that
\begin{equation}
    \sigma_k \geq C \| \nabla f(\x_k) \|, \quad C = \frac{2\kappa_{\min}}{L + 2c + 4\epsilon_f}.
\end{equation}
\end{lemma}
\end{mybox}

In this analysis, our goal is to show that the norm of the gradient decreases below a threshold
\begin{equation*}
    T_{\epsilon} = \inf\{ k \geq 0:\, \| \nabla f(\X_k) \| \leq \epsilon \}.
\end{equation*}
We assume that Assumption~\ref{ass:iterates} (i) holds by the choice of $\Sigma_{\max}$. We also know from Lemma~\ref{lemma:unsuccessful-step} that for $\| \nabla f(\X) \| > \epsilon$ and $\Sigma \leq C \epsilon$ then a successful step occurs, provided that estimates are accurate. Then following Lemma 4.10 from~\cite{paquette2018stochastic} we get that Assumption~\ref{ass:iterates} (ii) also holds, for $\Sigma_{\epsilon} = C \epsilon$. Based on the results of Theorem~\ref{theorem:stopping-time} and Lemma~\ref{lemma:unsuccessful-step}, we can now prove convergence for a nonconvex bounded function.

\begin{mybox}{gray}
\begin{theorem} \label{theorem:convergence-nonconvex}
Assume that the Assumptions of Theorem~\ref{theorem:pfi-k-nonconvex} hold with additionally $p_f > \frac{1}{2}$. Then to get $\| \nabla f(\X_k) \| \leq \epsilon$, the expected stopping time of Algorithm~\ref{algo:direct-search} is
\begin{equation*}
    \mathbb{E}[T_{\epsilon}] \leq \mathcal{O}(1) \frac{\kappa_{\min}^{-2}}{2 p_f - 1} (f(\X_0) - f^* + \Sigma_0^2)(L + c + \epsilon_f)^2 \frac{1}{\epsilon^2}.
\end{equation*}
\end{theorem}
\end{mybox}

Note that for the deterministic scenario where $\epsilon_f = l_f = 0$, the above bound matches known results of direct-search in the nonconvex case~\citep{vicente2013worst, konevcny2014simple}. We now establish faster convergence for a function $f$, additionally satisfying the PL condition, defined below.

\begin{definition} \label{def:PL-condition}
(Polyak-\L{}ojasiewicz Condition). A differentiable function $f: \R^n \to \R$ with the minimum value $f^* = \min_\x f(\x)$ is said to be $\mu$-Polyak-\L{}ojasiewicz ($\mu$-PL) if:
$$\frac{1}{2}\| \nabla f(\x) \|^2 \geq \mu (f(\x) - f^*).$$
\end{definition}

The PL condition is the weakest among a large family of function classes that include convex functions and other nonconvex ones~\citep{karimi2015linear}. Again we can guarantee convergence that closely matches results for deterministic direct-search under strong convexity, by proving that the number of iterations required to halve the distance to the optimum objective value is constant in terms of the accuracy $\epsilon$.

\begin{mybox}{gray}
\begin{theorem} \label{theorem:convergence-pl}
Let a function $f$ with a minimum value $f^*$ and satisfying the PL condition with a constant $\mu$ and Lipschitz continuous gradients with a constant $L$. Let also $f$ be $p_f$-probabilistically $\epsilon_f$-accurate, while also having bounded noise variance according to Assumption~\ref{ass:variance-condition} with constant $l_f$.  Then to get $\| \nabla f(\X_k) \| \leq \epsilon$, the expected stopping time of Algorithm~\ref{algo:direct-search} is
\begin{equation}
\begin{split}
    \mathbb{E}[T_\epsilon] \leq \mathcal{O}(1) &\frac{\kappa_{\min}^{-2} (c + L)^2}{(2p_f - 1) \mu} \left(1 + \frac{1}{c} \right)\\
    &\quad \log\left(\frac{L (f(\X_0) - f^*)}{\epsilon}\right).
\end{split}
\end{equation}
The constants $c$, $v$ and $p_f > \frac{1}{2}$ should satisfy
\begin{equation*}
    c > \max\{4 \epsilon_f, 2 \sqrt{2} l_f\} \text{, } \, \frac{p_f}{\sqrt{1 - p_f}} \geq \frac{4v l_f}{(1 - v) (1 - \gamma^{-2})}
\end{equation*}
\begin{equation*}
    \text{and }\frac{v}{1 - v} \geq \max \left \{\frac{1}{c - 2 \epsilon_f}(\gamma^2 - \frac{1}{\gamma^2}), \frac{72 \gamma^2}{c} \right \}.
\end{equation*}
\end{theorem}
\end{mybox}


\section{Algorithm \& Convergence Guarantees} \label{sec:convergence-min-max}

We now focus on the min-max problem presented in Eq.~\eqref{eq:min-max}. To proceed, we make the following standard assumptions regarding the smoothness of $f$.

\begin{assumption} \label{ass:lipschitz}
The function f is continuously differentiable in both $\x$ and $\y$ and there exist constants $L_{11}$, $L_{12}$, $L_{21}$ and $L_{22}$ such that for every $\x, \x_1, \x_2 \in \mathcal{X}$ and $\y, \y_1, \y_2 \in \mathcal{Y}$

\begin{align*}
    \| \nabla_{\x} f(\x_1, \y) - \nabla_{\x} f(\x_2, \y) \| &\leq L_{11} \| \x_1 - \x_2\|, \\
    \| \nabla_{\x} f(\x, \y_1) - \nabla_{\x} f(\x, \y_2) \| &\leq L_{21} \| \y_1 - \y_2\|, \\
    \| \nabla_{\y} f(\x_1, \y) - \nabla_{\y} f(\x_2, \y) \| &\leq L_{12} \| \x_1 - \x_2\|, \\
    \| \nabla_{\y} f(\x, \y_1) - \nabla_{\y} f(\x, \y_2) \| &\leq L_{22} \| \y_1 - \y_2\|.
\end{align*}
\end{assumption}


We require that the objective of the max-player satisfies the PL condition. 
\begin{assumption} \label{ass:pl}
There exists a constant $\mu > 0$ such that the function $- f(\x, \y)$ in problem~\eqref{eq:min-max} is $\mu$-PL for any $\x \in \mathcal{X}$.
\end{assumption}
Following prior works on PL games, e.g.~\cite{nouiehed2019solving}, we propose a sequential scheme for the updates of the two players presented in Algorithm~\ref{algo:min-max-DR} (for simplicity some of the algorithm's constants are not depicted). This multi-step algorithm solves the maximization problem up to some accuracy, and it then performs a single (successful) Direct-Search (DR) step for the minimization problem (see Algorithm~\ref{algo:one-step-DR}).

\begin{algorithm}
\SetAlgoLined
\caption{{\sc Min-Max-Direct-search}}
\label{algo:min-max-DR}
\KwIn{$f$: objective function}
\myinput{$(\x_0, \y_0)$: initial point}
\myinput{$\sigma_0$: initial step for the min problem}

\For{t = 1, \dots, T}{
$\y_{t}$ = Direct-search($-f(\x_{t - 1}, .), \y_{t - 1}$)

$\x_t, \sigma_t$ = One-Step-Direct-search \\
\qquad \qquad \qquad \qquad ($f(., \y_t), \x_{t - 1}, \sigma_{t - 1}$)
}
\Return  $(\x_{T},\y_{T})$.
\end{algorithm}


We formalize our Assumptions and our final result.

\begin{assumption} \label{ass:bounded}
The function $f$ is defined on the whole domain $\mathcal{X} \times \mathcal{Y} = \R^{|\mathcal{X}|} \times \R^{|\mathcal{Y}|}$. We also require $f$ to be bounded below for every $\y \in \mathcal{Y}$ and bounded above for every $\x \in \mathcal{X}$.
\end{assumption}

\begin{mybox}{gray}
\begin{theorem} \label{theorem:min-max-convergence}
Suppose that the objective function $f(\x, \y)$ satisfies Assumptions~\ref{ass:lipschitz}, \ref{ass:pl} and \ref{ass:bounded}. If the estimates are deterministic, then Algorithm~\ref{algo:min-max-DR} converges to an $\epsilon$-FNE within $\mathcal{O}(\epsilon^{-2}\log(\epsilon^{-1}))$ steps. When $f(\x, \y)$ is $\epsilon_x$-accurate with probability $p_x$ for every $\y$ satisfying assumptions of Theorem~\ref{theorem:convergence-nonconvex} and $\epsilon_y$-accurate with probability $p_y$ for every $\x$, satisfying assumptions of Theorem~\ref{theorem:convergence-pl}, then with a probability at least $\delta$, Algorithm~\ref{algo:min-max-DR} convergences and the expected number of steps to converge to reach an $\epsilon$-FNE is 
\begin{equation*}
\begin{split}
    \mathcal{O}&\Big(\frac{1}{(2p_x - 1)(2p_y - 1)}\epsilon^{-2}\Big(\log(\epsilon^{-1}) \\ &+ \left[\log\left( \frac{1 - p_x}{p_x}\right)\right]^{-1} \log\left(1 - e^{\frac{1}{(2p_x - 1)\epsilon^{-2}}\log\delta}\right)\Big)\Big).
\end{split}
\end{equation*}
\end{theorem}
\end{mybox}

Algorithm~\ref{algo:min-max-DR} performs in total $\mathcal{O}(\epsilon^{-2})$ updates for the minimization problem, and each minimization update requires $\mathcal{O}(\log(\epsilon^{-1}))$ updates for the maximization problem. The proof of Theorem~\ref{theorem:min-max-convergence} consists in showing that the maximization problem is solved with sufficient accuracy, for which we invoke the result of Theorem~\ref{theorem:convergence-pl}. We then proceed by showing that iteratively solving the minimization problem allows us to converge in terms of the min-max objective, which is done using the result of Theorem~\ref{theorem:convergence-nonconvex}.
We note that the sufficient decrease condition allows us to prove convergence for the last iterate instead of relying on the existence of an iterate $k$ in the whole sequence that satisfies the required inequalities (as proven in the corresponding gradient based method by~\cite{nouiehed2019solving}).

\begin{figure*}[t!]
    \centering
    \includegraphics[width=1\textwidth]{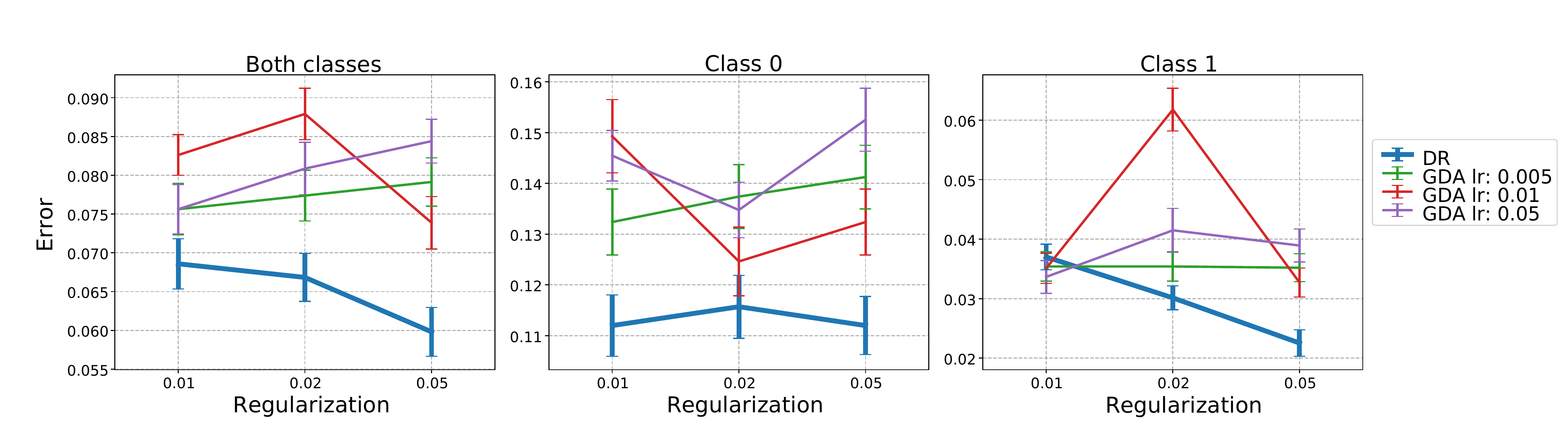}
    \caption{Zero-one loss for each method across classes. The term "lr" stands for different learning rates used. Error bars correspond to 20\% of the standard deviation across 10-fold cross validation.}
    \label{fig:robust-optimization}
\end{figure*}

\section{Experiments}
One advantage of direct-search methods is their abilities to explore the space of parameters. This however comes at the price of a high dependency to the size of the parameter space~\citep{vicente2013worst}. For nonconvex optimization problems in $\R^n$, the complexity of DS methods is of the order $\mathcal{O}(n^2)$ \citep{dodangeh2016optimal}. However, recent works by~\cite{gratton2015direct, bergou2019stochastic} have shown that replacing the sampling procedure from a PSS by one that correlates with the gradient direction probabilistically, it is possible to achieve a dependence of the order $\mathcal{O}(n)$. The sequential aspect of our method allows us to adopt this probabilistic perspective for the experiments to follow, thus lowering the computation cost.

\subsection{Robust Optimization}

Robustly-regularized estimators have been successfully used in prior work~\citep{namkoong2017variance} to deal with situations in which the empirical risk minimizer is susceptible to high amounts of noise. Formally, the problem of empirical risk minimization can be formulated as follows,
\begin{align}
    \min_{\pmb{\theta}} \sup_{P \in \mathcal{P}} [ f(\Xm; \pmb{\theta}, \mathcal{P}) = \{ &\mathbb{E}_{\mathcal{P}}[\mathit{l}(\Xm;  \pmb{\theta})]\,:\nonumber \\ 
    &\, \infdiv{\mathcal{P}}{\hat{\mathcal{P}}_n} \leq \frac{\rho}{n}\}],
    \label{eq:ERM}
\end{align}
where $\mathit{l}(\Xm;  \pmb{\theta})$ denotes the loss function, $\Xm$ the data and $\infdiv{\mathcal{P}}{\hat{\mathcal{P}}_n}$ a distance function that measures the divergence between the true data distribution $\mathcal{P}$ and the empirical data distribution $\mathcal{P}_n$. For the specific case of a binary classification problem, as for instance considered in~\cite{adolphs2018local}, Eq.~\eqref{eq:ERM} can be reformulated as
\begin{align*}
    \min_{\pmb{\theta}} \max_{\p} \{ &- \sum_{i=1}^n p_i[y_i \log (\hat{y}(\Xm_i;\pmb{\theta})) + (1 - y_i)  \\
    & \log (1 - \hat{y}(\Xm_i;\pmb{\theta}))] - \lambda \sum_{i=1}^n \left(p_i - \frac{1}{n} \right)^2 \},
\end{align*}
where $y_i$ and $\hat{y}(\Xm_i;\pmb{\theta})$ correspond to the true and the predicted class of data point $\Xm_i$ and $\lambda > 0$ controls the amount of regularization. Note that the aforementioned function is strongly-concave w.r.t $\p$ (i.e. it satisfies our PL assumption) and can thus be solved efficiently. We consider this optimization problem on the Wisconsin breast cancer data set\footnote{\url{https://archive.ics.uci.edu/ml/datasets/Breast+Cancer+Wisconsin+(Diagnostic)}}, comparing the performance between our proposed direct-search method and GDA, using the same neural network as classifier. The zero-one loss is shown in Fig.~\ref{fig:robust-optimization} which clearly shows that our algorithm can consistently outperform GDA for different choices of regularization parameters.

\subsection{Categorical Data}
Generative Adversarial Networks~\citep{goodfellow2014generative} are formulated as the saddle point problem:

\begin{align}
    \min_{\x} \max_{\y} f(\x, \y) = &\mathbb{E}_{\pmb{\theta} \sim p_{data}} [\log \D_{\y}(\pmb{\theta})] \nonumber \\
    &+ \mathbb{E}_{\z \sim p_{\z}} [\log (1 - \D_{\y}(\G_{\x}(\z)))] \nonumber,
\end{align}
where $\D_{\y} : \R^n \to [0, 1]$ and $\G_{\x} : \R^m \to \R^n$ are the discriminator and generator networks. Although GANs have been used in a wide variety of applications~\citep{goodfellow2016nips}, very few approaches can deal with discrete data. The most severe impeding factor in such settings is the non existence of the gradient due to the non-smooth nature of the objective function. One advantage of direct-search techniques over gradient-based methods is that they can be used in such a context where gradients are not accessible. In some cases, we note that $\ell_2$ regularization can be used to increase the smoothness constant of the objective function.
We illustrate the performance of our direct-search algorithm on a simple example consisting of correlated categorical data, in Figure~\ref{fig:distributions-discrete-gaussian-dr}. For a more detailed discussion and more experimental results we refer the reader to the Appendix.
\begin{figure*}[ht!]
    \centering
    \includegraphics[width=1\textwidth]{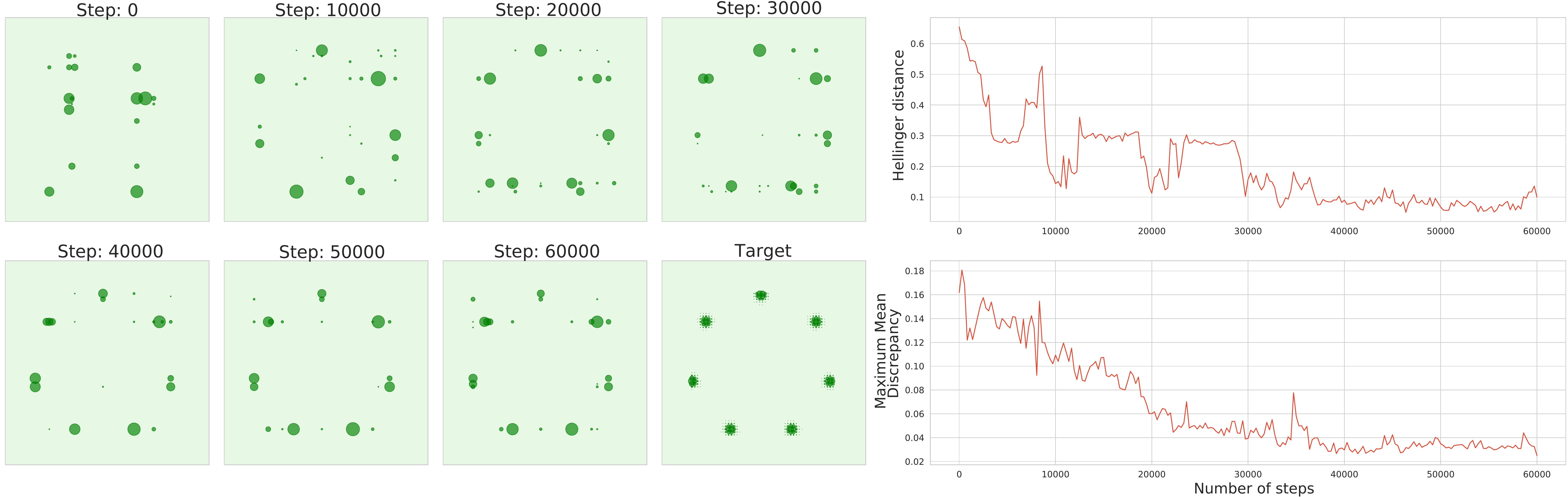}
    \caption{Learning a discretized mixture of Gaussian processes using direct-search methods. Both the Hellinger distance and maximum mean discrepancy decrease as DR learns the modes of the distribution.}
    \label{fig:distributions-discrete-gaussian-dr}
\end{figure*}

Scaling direct-search to higher dimensions still remains an active area of research, where recent developments include guided search~\citep{maheswaranathan2018guided} and projection-based approaches~\citep{wang2016bayesian}. In this work, we focus on the theoretical guarantees or our algorithm in the stochastic min-max setting. While we demonstrate a good empirical behavior on relatively small-scale problems, scaling our algorithm to large-scale problems will require further modifications to improve its scalability.


\section{Conclusion}

We presented and proved convergence results for a direct-search method in a stochastic minimization setting for both nonconvex and PL objective functions. We then extended these results to prove convergence for min-max objective functions, where the objective of the max-player satisfies the (PL) condition, while the min-player objective is nonconvex. Our experimental results establish that direct-search can outperform traditionally adopted optimization schemes, while also presenting a promising alternative for categorical settings. A potential direction for future work is to improve the scalability of our algorithm in order to run it on large-scale problems, such as adversarial poisoning attacks on benchmark computer vision datasets. Additional extensions of our work include the use of momentum to accelerate convergence as in~\cite{gidel2018negative} or developing an optimistic variant of our algorithm as in~\cite{daskalakis2017training, daskalakis2018limit}.

\section{Acknowledgements}

Sotiris Anagnostidis is supported by by the Onassis Foundation - Scholarship ID: F ZP 002-1/2019-2020.

{\small
\bibliographystyle{plainnat}
\bibliography{bibliography}
}

\newpage

\begin{appendices}

\vspace{-10mm}
\section{Algorithms} \label{app:algos}

We present omitted algorithms. Algorithm~\ref{algo:one-step-DR} depicts the updates for the minimization problem. At each outer iteration of Algorithm~\ref{algo:min-max-DR}, a single successful step for the minimization problem is performed. In contrast to standard direct-search algorithms, we do not increase the step size parameter immediately after a successful step, but instead, before the start of the next search for a new successful step, to simplify notation for the upcoming proofs.

\begin{algorithm}
\SetAlgoLined
\caption{{\sc One-Step-Direct-search($f, \x_0, \sigma_0$)}}
\label{algo:one-step-DR}
\KwIn{$f$: objective function, with $f_k$ it's estimate at step $k$}
\myinput{$\x$: initial point}
\myinput{$\sigma_0$: step size value}
$c$: forcing function constant\\
$\gamma > 1$: step size update parameter\\

Create the Positive Spanning Set $\D$ for the variables $\x$.\\
Update $\sigma_{1} = \min\{\gamma \sigma_0, \sigma_{\max} \}$ as last update was successful.

\For{k = 1, \dots}{
\textbf{1. Offspring generation:}\\
Generate the points $$ \x^i = \x + \sigma_k \d^i, \quad \forall \d^i \in \D.$$ \\
\textbf{2. Parent Selection:}\\
Choose $\x^{\prime} = \argmin_i f_k(\x^i)$.\\

\textbf{3. Sufficient Decrease:}
\uIf{ $f_k(\x^{\prime}) < f_k(\x) - \rho(\sigma_k)$ }{
    (Iteration is successful)\\
    \Return $\x^{\prime}, \sigma_{k}$.
}\Else{
    (Iteration is unsuccessful)\\
    Decrease step size $\sigma_{k+1} = \gamma^{-1} \sigma_k$.
}}
\end{algorithm}

\section{Proofs of Section~\ref{sec:convergence-min}} \label{app:proofs-min}

\subsection{Proof of Lemma~\ref{lemma:bound-absolute}}

\begin{proof}
The result follows by applying Holder's inequality.

\begin{equation*}
    \mathbb{E} \left[\frac{1_{J_k^c}|F_k^0 - f(\X_k)|}{l_f \Sigma_k^2} \mid \F_{k-1} \right] \leq \left( \mathbb{E}[1_{J_k^c}|\F_{k-1}] \right)^{1/2} \left( \mathbb{E} \left[ \frac{|F_k^0 - f(\X_k)|^2}{l_f^2 \Sigma_k^4}\right]\right)^{1/2}.
\end{equation*}

By Assumption~\ref{ass:variance-condition}, it holds that $\left( \mathbb{E} \left[ \frac{|F_k^0 - f(\x_k)|^2}{l_f^2 \Sigma_k^4}\right]\right)^{1/2} \leq 1$ and the result follows. Following the same steps, the second inequality of the Lemma holds as well.

\end{proof}

\subsection{Proof of Theorem~\ref{theorem:pfi-k-nonconvex}}

\begin{proof}
We begin by taking separate cases according to if the estimates are accurate or not and if the steps of Algorithm~\ref{algo:direct-search} are successful or not. We use $1_{\text{Succ}_k}$ to denote the event that step $k$ is successful.

\textbf{Case 1: Accurate estimates.}

\begin{itemize}
    \item \textbf{Successful step.}
    
    At a successful step with accurate estimates we have that:
    
    \begin{align*}
        1_{\text{Succ}_k} 1_{J_k} &(f(\X_{k + 1}) - f(\X_k)) \\ &= 1_{\text{Succ}_k} 1_{J_k} (f(\X_{k + 1}) - f_k(\X_{k + 1}) + f_k(\X_{k + 1}) - f_k(\X_k) + f_k(\X_k) - f(\X_k)) \\ 
        &\leq 1_{\text{Succ}_k} 1_{J_k}(- ( c - 2\epsilon_f) \Sigma_k^2).
    \end{align*}

    Therefore

    \begin{align*}
        1_{\text{Succ}_k} 1_{J_k} &(\Phi_{k + 1} - \Phi_k) \\ &= 1_{\text{Succ}_k} 1_{J_k} (v(f(\X_{k + 1}) - f(\X_k)) + (1 - v) \Sigma_{k + 1}^2 - (1 - v) \Sigma_k^2) \\
        &\leq 1_{\text{Succ}_k} 1_{J_k}(- v (c - 2 \epsilon_f) \Sigma_k^2 + (1 - v) (\gamma^2 - 1) \Sigma_k^2).
    \end{align*}

    \item \textbf{Unsuccessful step.}

    \begin{align*}
        1_{\text{Succ}_k^c} 1_{J_k} (\Phi_{k + 1} - \Phi_k) &= 1_{\text{Succ}_k^c} 1_{J_k}((1 - v) \Sigma_{k + 1}^2 - (1 - v) \Sigma_k^2) \\
        &= 1_{\text{Succ}_k^c} 1_{J_k}(- (1 - v) (1 - \frac{1}{\gamma^2}) \Sigma_k^2).
    \end{align*}
   
\end{itemize}

Combining the above results and given that

\begin{equation*}
    \frac{v}{1 - v} \geq \frac{1}{c - 2 \epsilon_f}(\gamma^2 - \frac{1}{\gamma^2}) \implies - v(c - 2\epsilon_f) + (1 - v)(\gamma^2 - 1) \leq - (1 - v) (1 - \frac{1}{\gamma^2}),
\end{equation*}

in the case of accurate estimates we have

\begin{equation}
    \mathbb{E}[1_{J_k}(\Phi_{k + 1} - \Phi_k) \mid \F_{k - 1}] \leq - p_f (1 - v) (1 - \frac{1}{\gamma^2}) \Sigma_k^2.
\end{equation}

\textbf{Case 2: Inaccurate estimates.}

\begin{itemize}
    \item  \textbf{Successful step.}

    \begin{align*}
        1_{\text{Succ}_k} 1_{J_k^c} (\Phi_{k + 1} - \Phi_k) &= 1_{\text{Succ}_k} 1_{J_k^c} (v(f(\X_{k + 1}) - f(\X_k)) + (1 - v) \Sigma_{k + 1}^2 - (1 - v) \Sigma_k^2) \\
        &= 1_{\text{Succ}_k} 1_{J_k^c} (v(f(\X_{k + 1}) - f_k(\X_{k + 1}) + f_k(\X_{k + 1}) - f_k(\X_k)  \\ &\quad \quad \quad + f_k(\X_k) - f(\X_k)) + (1 - v) \Sigma_{k + 1}^2 - (1 - v) \Sigma_k^2) \\
        &\leq 1_{\text{Succ}_k} 1_{J_k^c} (- v c \Sigma_k^2 + v|f(\X_{k + 1}) - f_k(\X_{k + 1})| + v |f(\X_{k}) - f_k(\X_{k})| \\ &\quad \quad \quad - (1 - v) (\gamma^2 - 1) \Sigma_k^2),
    \end{align*}
    
    where we will later bound terms $|f(\X_{k + 1}) - f_k(\X_{k + 1})|, |f(\X_{k}) - f_k(\X_{k})|$ using Lemma~\ref{lemma:bound-absolute}.
        
    \item \textbf{Unsuccessful step.}
    
    As before:

    \begin{align*}
        1_{\text{Succ}_k^c} 1_{J_k^c} (\Phi_{k + 1} - \Phi_k) &= 1_{\text{Succ}_k^c} 1_{J_k^c} ((1 - v) \Sigma_{k + 1}^2 - (1 - v) \Sigma_k^2) \\
        &= 1_{\text{Succ}_k^c} 1_{J_k^c} (- (1 - v) (1 - \frac{1}{\gamma^2}) \Sigma_k^2).
    \end{align*}
\end{itemize}


In total for inaccurate estimates and by using Assumption~\ref{ass:variance-condition} and Lemma~\ref{lemma:bound-absolute}

\begin{equation}
    \mathbb{E}[1_{J_k^c}(\Phi_{k + 1} - \Phi_k) \mid \F_{k - 1}] \leq  2v(1 - p_f)^{1/2} l_f \Sigma_k^2.
\end{equation}

Finally, integrating both successful and unsuccessful iterations 

\begin{align*}
    \mathbb{E}[\Phi_{k + 1} - \Phi_k \mid \F_{k - 1}] &\leq - p_f (1 - v) (1 - \frac{1}{\gamma^2}) \Sigma_k^2 + 2v(1 - p_f)^{1/2} l_f \Sigma_k^2 \\
    &\leq - p_f (1 - v) (1 - \frac{1}{\gamma^2}) \frac{\Sigma_k^2}{2},
\end{align*}

for our requirement of $\frac{p_f}{\sqrt{1 - p_f}} \geq \frac{4v l_f}{(1 - v) (1 - \gamma^{-2})}$.


\end{proof}

\subsection{Proof of Lemma~\ref{lemma:unsuccessful-step}}

\begin{proof}
Similar to~\cite{conn2009introduction}, for an unsuccessful step with accurate estimates, we have that for some $\d_k \in \D$

\begin{equation}
\cm(\D) \| \nabla f(\x_k) \| \| \d_k \| \leq - \nabla f(\x_k)^\top \d_k.
\label{eq:d_k_positive_spanning_set}
\end{equation}

By the mean value theorem, for some $\eta_k \in [0,1]$,
\begin{equation*}
f(\x_k + \sigma_k \d_k) - f(\x_k) = \sigma_k \nabla f(\x_k + \eta_k \sigma_k \d_k)^\top \d_k.
\end{equation*}

Since $k$ is the index of an unsuccessful iteration,
\begin{equation*}
f_k(\x_k + \sigma_k \d_k) - f_k(\x_k) + \rho(\sigma_k) \geq 0
\end{equation*}

and since estimates are accurate

\begin{align*}
    f(\x_k + \sigma_k \d_k) - f(\x_k) &= f(\x_k + \sigma_k \d_k) - f_k(\x_k + \sigma_k \d_k) \\ & \quad \quad \quad + f_k(\x_k + \sigma_k \d_k) - f_k(\x_k) + f_k(\x_k) - f(\x_k) \\
    &\geq  - \epsilon_f \sigma_k^2 - \rho(\sigma_k) - \epsilon_f \sigma_k^2 \\ 
    &= - (c + 2 \epsilon_f) \sigma_k^2.
\end{align*}

Combining the above equations,

\begin{align}
&\sigma_k \nabla f(\x_k + \eta_k \sigma_k \d_k)^\top \d_k + (c + 2 \epsilon_f) \sigma_k^2 \geq 0 \nonumber \\
\implies &  \nabla f(\x_k + \eta_k \sigma_k \d_k)^\top \d_k + (c + 2 \epsilon_f) \sigma_k \geq 0 \nonumber \\
\implies &  - \nabla f(\x_k)^\top \d_k \leq (\nabla f(\x_k + \eta_k \sigma_k \d_k) - \nabla f(\x_k))^\top \d_k + (c + 2 \epsilon_f) \sigma_k,
\end{align}

where in the last inequality, we subtracted $\nabla f(\x_k)^\top \d_k$ from both sides.

Finally, Eq.~\eqref{eq:d_k_positive_spanning_set} implies
\begin{align*}
\cm(\D) \| \nabla f(\x_k) \| \| \d_k \| &\leq (\nabla f(\x_k + \eta_k \sigma_k \d_k) - \nabla f(\x_k))^\top \d_k + (c + 2 \epsilon_f) \sigma_k \\
\implies \cm(\D) \| \nabla f(\x_k) \| &\leq \| \nabla f(\x_k + \eta_k \sigma_k \d_k) - \nabla f(\x_k) \| + (c + 2 \epsilon_f) \sigma_k \\
&\leq \frac{L}{2} \sigma_k + (c + 2 \epsilon_f) \sigma_k.
\end{align*}
\end{proof}

\subsection{Proof of Theorem~\ref{theorem:convergence-nonconvex}}

\begin{proof}
By Theorem~\ref{theorem:pfi-k-nonconvex} and Lemma 4.10 from~\cite{paquette2018stochastic} we get that Assumption~\ref{ass:iterates} is satisfied, for $\Sigma_{\epsilon} = C \epsilon$. Then by an application of Theorem~\ref{theorem:stopping-time} we get

\begin{equation*}
    \mathbb{E}[T_{\epsilon}] \leq \frac{p_f}{2 p_f - 1} \frac{v(f(\X_0) - f^*) + (1 - v) \Sigma_0^2}{p_f(1 - v)(1 - \frac{1}{\gamma^2}) \frac{C^2 \epsilon^2}{2}}.
\end{equation*}

The result follows.

\end{proof}

\subsection{Proof of Theorem~\ref{theorem:convergence-pl}}

We will also use the additional result holding for any function with Lipschitz-continuous gradients.

\begin{lemma} \label{lemma:bound-norm-by-distance}
Let $f: \x \in \R^n \to \R$ be a continuous differentiable function with Lipschitz continuous gradient with a constant $L$ and a minimum value achieved for $\x^*$. Then

\begin{equation}
    f(\x) - f(\x^*) \geq \frac{1}{2L} \| \nabla f(\x) \|^2.
\end{equation}
\end{lemma}

\begin{proof}
By smoothness and for $\y = \x - \frac{1}{L} \nabla f(\x)$ we have

\begin{align*}
    f(\x) - f(\x^*) &\geq f(\x) - f(\y) \\
    &\geq \langle \nabla f(\x), \x - \y \rangle -\frac{L}{2} \| \y - \x \|^2 \\
    &= \frac{1}{L} \| \nabla f(\x) \|^2 - \frac{1}{2 L} \| \nabla f(\x) \|^2 \\
    &= \frac{1}{2L} \| \nabla f(\x) \|^2.
\end{align*}

\end{proof}

We can now proceed with the proof of Theorem~\ref{theorem:convergence-pl}.

\begin{proof}
We note that for the conditions on the constants $c$, $v$ and $p_f$, requirements of Theorem~\ref{theorem:pfi-k-nonconvex} are also satisfied. We define as $T_i = \inf\{ k \geq 0:\, f(\X_k) - f^* \leq \frac{f(\X_0) - f^*}{2^i}\}$, with $T_0 = 0$. We will also use the random variable $\Lambda_i = T_i - T_{i - 1}$.

We will assume without loss of generality that 

\begin{equation}
    \Sigma_0^2 \leq \frac{9 \gamma^2}{c} (f(\X_0) - f^*) \triangleq A (f(\X_0) - f^*),
\end{equation}
for $A = \frac{9 \gamma^2}{c}$. We apply Theorem~\ref{theorem:stopping-time}. Given that $f(\X_{T_{i - 1}}) - f^* \leq \frac{f(\X_{T_{0}}) - f^*}{2^{i - 1}}$, and that $f(\X_{k}) - f^* > \frac{f(\X_{T_{0}}) - f^*}{2^{i}}$ for $k \in [T_{i - 1}, T_i)$ (possibly an empty set), Lemma~\ref{lemma:unsuccessful-step} and the Definition~\ref{def:PL-condition}, then for step sizes $\Sigma_k^2 \leq C^2 \mu \frac{f(\X_0) - f^*}{2^{i - 1}}$ and accurate estimates, steps are successful. Then by Theorem~\ref{theorem:stopping-time} for an application of the results from Theorem~\ref{theorem:pfi-k-nonconvex} as before, we have

\begin{align}
    \mathbb{E}[\Lambda_i \mid \F_{T_{i - 1} - 1}] &\leq \frac{p_f}{2p_f - 1} \frac{v(f(\X_{T_{i - 1}}) - f^*) + (1 - v) \Sigma_{T_{i - 1}}^2}{p_f (1 - v) (1 - \gamma^{-2}) \frac{1}{2} C^2 \mu \frac{f(\X_0) - f^*}{2^{i - 1}}} \nonumber \\
    &= \frac{2}{(2p_f - 1) (1 - \gamma^{-2})  C^2 \mu} \left( \frac{v}{1-v} \frac{f(\X_{T_{i - 1}}) - f^*}{  \frac{f(\X_0) - f^*}{2^{i - 1}}}  + \frac{\Sigma_{T_{i - 1}}^2}{  \frac{f(\X_0) - f^*}{2^{i - 1}}} \right) \nonumber \\
    &\leq \frac{2}{(2p_f - 1) (1 - \gamma^{-2})  C^2 \mu} \left( \frac{v}{1-v} \frac{\frac{f(\X_0) - f^*}{2^{i - 1}}}{  \frac{f(\X_0) - f^*}{2^{i - 1}}}  + \frac{\Sigma_{T_{i - 1}}^2}{  \frac{f(\X_0) - f^*}{2^{i - 1}}} \right) \nonumber \\
    &= \frac{2}{(2p_f - 1) (1 - \gamma^{-2})  C^2 \mu} \left( \frac{v}{1-v} + \frac{\Sigma_{T_{i - 1}}^2}{  \frac{f(\X_0) - f^*}{2^{i - 1}}} \right)
\end{align}

We will further show with induction that $\mathbb{E} [\Sigma_{T_i}^2] \leq A \frac{f(\X_0) - f^*}{2^{i}}$. As a result

\begin{align}
    \mathbb{E}[\Lambda_i] &\leq \frac{2}{(2p_f - 1) (1 - \gamma^{-2})  C^2 \mu} \left( \frac{v}{1-v}   + \frac{\mathbb{E}[\Sigma_{T_{i - 1}}^2]}{  \frac{f(\X_0) - f^*}{2^{i - 1}}} \right) \nonumber \\
    &\leq \frac{2}{(2p_f - 1) (1 - \gamma^{-2})  C^2 \mu} \left( \frac{v}{1-v} + A \right).
\end{align}

The final complexity will be:

\begin{align}
    \mathbb{E}[T_{\ceil{\log\frac{2L(f(\X_0) - f^*)}{\epsilon^2}}}] &= \mathbb{E}[\Lambda_1 + \Lambda_2 + \dots + \Lambda_{\ceil{\log\frac{2L(f(\X_0) - f^*)}{\epsilon^2}}}] \nonumber \\
    &\leq \frac{2}{(2p_f - 1) (1 - \gamma^{-2})  C^2 \mu} \left( \frac{v}{1-v} + A \right) \ceil{\log\left(\frac{2L(f(\X_0) - f^*)}{\epsilon^2}\right)}.
\end{align}
Getting $\| \nabla f(\x_0) \|^2 \leq 2L(f(\x_0) - f^*)$ from Lemma~\ref{lemma:bound-norm-by-distance}, the result follows.

It remains to show the result that $\mathbb{E} [\Sigma_{T_i}^2] \leq A \frac{f(\X_0) - f^*}{2^{i}}$. By assumption, as aforementioned, it holds for $T_0$. We then assume that it holds for $T_{i - 1}$ and show that it also holds for $T_i$. For each $T_i$, the last step $T_i - 1$ was a successful one as the parameter $\X$ was updated to satisfy the goal $f(\X_{T_i}) - f^* \leq \frac{f(\X_0) - f^*}{2^{i}}$. As in Theorem~\ref{theorem:pfi-k-nonconvex} we differentiate between the events of this step being accurate or not. 

Since we have a successful step $f_{T_i - 1}(\X_{T_i}) - f_{T_i - 1}(\X_{T_i - 1}) \leq - c \Sigma_{T_i - 1}^2$. Then

\begin{align}
    f(\X_{T_i}) - f(\X_{T_i - 1}) &= f(\X_{T_i}) - f_{T_i - 1}(\X_{T_i}) + \nonumber \\ &\quad \quad f_{T_i - 1}(\X_{T_i}) - f_{T_i - 1}(\X_{T_i - 1}) + f_{T_i - 1}(\X_{T_i - 1}) - f(\X_{T_i - 1})
\end{align}

We denote with $p_{\text{Acc}}$ the probability of this last step being accurate. Note that this is not the same as $p_f$ as we are conditioning on a successful step. Then we distinguish the two cases. 

\begin{itemize}
    \item  \textbf{Accurate estimates.}
    
    By Assumption~\ref{def:accurate} we get that
    
    \begin{equation}
        1_{\text{Acc}} (f(\X_{T_i}) - f(\X_{T_i - 1})) \leq 1_{\text{Acc}} (- (c - 2 \epsilon_f) \Sigma_{T_i - 1}^2).
    \end{equation}

    \item \textbf{Inaccurate.}
    
    In this case, similarly to the proof of Lemma~\ref{lemma:bound-absolute} we get
    
    \begin{equation}
        \mathbb{E}[1_{\text{Acc}}^c (f(\X_{T_i}) - f(\X_{T_i - 1})) \mid \F_{T_i - 2}] \leq - (1 - p_{\text{Acc}}) c \Sigma_{T_i - 1}^2 + 2 \sqrt{1 - p_{\text{Acc}}} l_f \Sigma_{T_i - 1}^2.
    \end{equation}

\end{itemize}

Combining the above cases, we get

\begin{align}
    \mathbb{E}[f(\X_{T_i}) - f(\X_{T_i - 1})\mid \F_{T_i - 2}] &\leq - p_{\text{Acc}} (c - 2 \epsilon_f) \Sigma_{T_i - 1}^2 - (1 - p_{\text{Acc}}) c \Sigma_{T_i - 1}^2 + 2 \sqrt{1 - p_{\text{Acc}}} l_f \Sigma_{T_i - 1}^2 \nonumber \\
    &\leq - c \Sigma_{T_i - 1}^2 (1 - \frac{p_{\text{Acc}}}{2} - \sqrt{\frac{1 - p_{\text{Acc}}}{2}}), \,\text{ for } c > \max\{4 \epsilon_f, 2 \sqrt{2} l_f\} \nonumber \\
    &\leq - \frac{c}{4} \Sigma_{T_i - 1}^2 \nonumber \\
    \implies \Sigma_{T_i - 1}^2 &\leq \frac{4 \mathbb{E}[f(\X_{T_i - 1}) - f^*| \F_{T_i - 2}]}{c}. \label{eq:bound-sigma}
\end{align}

Furthermore, by Theorem~\ref{theorem:pfi-k-nonconvex} we get that

\begin{align}
   \mathbb{E}[\Phi_{T_i - 1} \mid \F_{T_{i - 1} - 1}] &\leq \Phi_{T_{i - 1}} \nonumber \\
    \mathbb{E}[v(f(\X_{T_i - 1}) - f^*) \mid \F_{T_{i - 1} - 1}] &\leq v(f(\X_{T_{i - 1}}) - f^*) + (1 - v) \Sigma_{T_{i - 1}}^2 \nonumber \\
    \mathbb{E}[f(\X_{T_i - 1}) - f^*] &\leq \frac{f(\X_{0}) - f^*}{2^{i - 1}} + \frac{(1 - v)}{v} \mathbb{E}[ \Sigma_{T_{i - 1}}^2] \nonumber \\
    &\leq \frac{f(\X_{0}) - f^*}{2^{i - 1}} + \frac{(1 - v)}{v} A \frac{f(\X_{0}) - f^*}{2^{i - 1}} \nonumber \\
    &\leq \frac{f(\X_{0}) - f^*}{2^{i}} 2 (1 + \frac{1 - v}{v} A) \label{eq:bound-hx}.
\end{align}

By combining \eqref{eq:bound-sigma} and \eqref{eq:bound-hx} and using the law of iterated expectation we have that

\begin{align}
    \mathbb{E}[\Sigma_{T_i}^2] &= \mathbb{E}[\gamma^2 \Sigma_{T_i - 1}^2] \nonumber \\ 
    &\leq \frac{f(\X_{0}) - f^*}{2^{i}} \frac{8 \gamma^2}{c
    } (1 + \frac{1 - v}{v} A) \nonumber \\
    &\leq A \frac{f(\X_{0}) - f^*}{2^{i}},
\end{align}
for $\frac{v}{1 - v} \geq \frac{72 \gamma^2}{c}$ and $A = \frac{9 \gamma^2}{c}$. The proof is complete.

\end{proof}

\section{Proofs of Section~\ref{sec:convergence-min-max}} \label{app:proofs-min-max}

We first present some additional results, required for our proof.

From~\cite{karimi2015linear}, for a function that satisfies the PL condition, it additionally satisfies the Quadratic Growth (GQ) condition. 

\begin{lemma} \label{lemma:QG-condition}
A differentiable function $f$ that satisfies the PL condition with parameter $\mu$, also satisfies the QG condition with parameter $4\mu$:

$$f(\x) - f^* \geq 2 \mu \| \x^* - \x \|^2,$$

where $\x^*$ belongs to the solution set $\mathcal{X}^*$.
\end{lemma}



Based on the previous Lemma, we can easily prove the following result.

\begin{lemma} \label{lemma:pl-epsilon}
Let a differentiable $\mu$-PL function $f$ and also $\x^* \in \argmin_{\x} f(\x)$. If we know that $\| \nabla f(\x) \| \leq \epsilon$ then:

$$\| \x - \x^* \| \leq \frac{1}{2 \mu} \epsilon.$$
\end{lemma}

\begin{proof}
By Lemma~\ref{lemma:QG-condition} and the definition of the PL condition we have that:

$$\| \x - \x^* \| \leq \sqrt{\frac{1}{2\mu} (f(\x) - f^*)} \leq \frac{1}{2\mu} \| \nabla f(\x) \| \leq \frac{1}{2 \mu} \epsilon.$$
\end{proof}

\begin{lemma} \label{lemma:min-max-connection}
(Lemma A.3 from~\cite{nouiehed2019solving}) Assume that $ - f(\x, \y)$ for a specific $\x$, is a class of $\mu$-PL functions in $\y$. Define the set of optimal solutions $\mathbbm{Y}(\x) = \argmax_{\y} f(\x, \y)$. Then for every $\x_1, \x_2 \in \mathcal{X}$ and $\y_1^* \in \mathbbm{Y}(\x_1), \y_2^* \in \mathbbm{Y}(\x_2)$ it holds that:

$$\| \y_1^* - \y_2^*\| \leq L_{xy} \| \x_1 - \x_2 \|,$$

where we denote with $L_{xy} = \frac{L_{12}}{2 \mu}$.
\end{lemma}

Next, we will need to establish a lower bound on the step size $\Sigma$. In the deterministic case, Lemma~\ref{lemma:unsuccessful-step} establishes such a lower bound for unsuccessful steps, guaranteeing that if $\Sigma = \Sigma_{\epsilon}$, then $\| \nabla f(\x) \| \leq \epsilon$. However, in the stochastic case, inaccurate steps may occur. We want to ensure a lower bound on the step size parameter with high probability.

\begin{figure}[t!]
    \centering
    \includegraphics[width=0.8\textwidth]{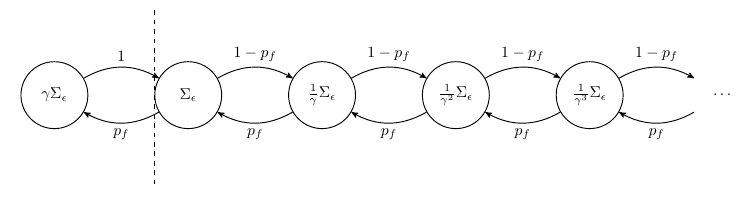}
    \caption{Worst case scenario for step sizes. Ignoring steps for $\Sigma > \Sigma_{\epsilon}$, it corresponds to a biased reflected random walk. The dotted line indices the barrier at position 0, indicating a step size of $\Sigma_{\epsilon}$.}
    \label{fig:markov-chain}
\end{figure}

To do so, we will consider the worst-case scenario where step sizes get as small as possible. This corresponds to the case where for all step sizes $\Sigma > \Sigma_{\epsilon}$, unsuccessful steps occur. So do all of the inaccurate estimates, with probability $1 - p_f$. For convenience, we will ignore steps above the value $\Sigma_{\epsilon}$ since we only require a bound. This corresponds to a random walk with a reflection barrier at position $0$ (which corresponds to the step size $\Sigma_{\epsilon}$) and an increment probability $1 - p_f$, where $p_f$ is the probability of accurate estimates. We, therefore, use the following Lemma to get a probabilistic lower bound on the step sizes.


\begin{lemma}  \label{lemma:random-walk}
Let a random walk starting at position 0, with a reflection barrier at position $0$ and a transition probability matrix 
\begin{equation*}
    \begin{bmatrix}
    p_f & 1 - p_f & & & \\
    p_f & 0 & 1 - p_f & & \\
    & p_f & 0 & 1 - p_f & \\
    & & \ddots & \ddots & \ddots \\
    \end{bmatrix}
\end{equation*}
for $p_f > \frac{1}{2}$. Then for $k \geq \frac{\log(1 - e^{\frac{1}{n}\log(\delta)})}{\log\left(\frac{1 - p_f}{p_f}\right)} - 1$, the random walk of length $n$, stays confined within the space $[0, k]$ with a probability at least $\delta > 0$.
\end{lemma}
\begin{proof}
Let a random walk $S_n = \max\{ S_{n - 1} + X_n, 0\}$, with $S_0 = 0$ and $\mathsf{P}(X_{n} = 1) = 1 - p_f$, $\mathsf{P}(X_{n} = -1) = p_f$, for $p_f > \frac{1}{2}$. The probability that the random walk stays until position $k$, $\mathsf{P}(S_i \leq k, \, \forall{i \leq n})$, is bounded below by the probability of $n$ randomly chosen points from the stationary distribution to be at positions lower or equal to k.

Let us denote with $p_{i,n}$ the probability that the random walk is at position $i$ after $n$ total steps. We first prove by induction that 

\begin{equation}
    p_{i,n} \geq \frac{p_f}{1 - p_f} p_{i + 1, n}. \label{eq:increase-n-inequality}
\end{equation}

It obviously holds for $n = 0$, as $p_{0,0} = 1$ and $p_{i, 0} = 0,\, \forall i \geq 1$. Assume that it holds for $n$. As shown in Fig.~\ref{fig:markov-chain}, with probability $(1-p_f)$, position $i$ is incremented, therefore for $i \geq 1$
\begin{align*}
    p_{i, n+1} &= p_{i - 1, n} (1 - p_f) + p_{i + 1,n} p_f \\
    &\geq \frac{p_f}{1 - p_f} p_{i, n} (1 - p_f) + \frac{p_f}{1 - p_f} p_{i + 2, n} p_f, \text{ by induction} \\
    &= \frac{p_f}{1 - p_f} (p_{i, n} (1 - p_f) + p_{i + 2, n} p_f) \\
    &= \frac{p_f}{1 - p_f} p_{i+1, n + 1}
\end{align*}
    
and for $i=0$

\begin{align*}
    p_{0, n + 1} &= p_{0, n} p_f + p_{1,n} p_f \\
    &\geq p_{0,n} p_f + \frac{p_f}{1 - p_f} p_{2, n} p_f , \text{ by induction} \\
    &= \frac{p_f}{1 - p_f} (p_{0, n} (1 - p_f) + p_{2, n} p_f) \\
    &= \frac{p_f}{1 - p_f} p_{1, n + 1}.
\end{align*}

Let us now consider the probability that the random walk resides in the first $k$ positions. Then:

\begin{align*}
    \sum_{i = 0}^k p_{i, n + 1} &= p_{0, n} p_f + p_{1, n} p_f + \sum_{i = 1}^k (p_{i - 1, n} (1 - p_f) + p_{i + 1, n} p_f) \\
    &= \sum_{i=0}^k p_{i, n} - p_{k, n} (1 - p_f) + p_{k + 1, n} p_f \\
    &\leq \sum_{i=0}^k p_{i, n}, \, \text{ by } \eqref{eq:increase-n-inequality},
\end{align*}
where the equality in the second line is due to the terms telescoping in the sum in the first line.

As a result, we can lower bound the probability $\sum_{i = 0}^k p_{i, n}$ with the corresponding one for $n \to \infty$, which corresponds to a stationary distribution. Also

\begin{align*}
    \mathsf{P}(S_i \leq k \mid S_{i - 1} \leq k) &= \mathsf{P}(S_i \leq k \mid S_{i - 1} = k) \,\mathsf{P}(S_{i - 1} = k) + \mathsf{P}(S_i \leq k \mid S_{i - 1} < k)\, \mathsf{P}(S_{i - 1} < k) \\
    &= p_f\, \mathsf{P}(S_{i - 1} = k) + \mathsf{P}(S_{i - 1} < k) \geq p_f
\end{align*}

and

\begin{align*}
    \mathsf{P}(S_i \leq k \mid S_{i - 1} > k) &= \mathsf{P}(S_i \leq k \mid S_{i - 1} = k + 1)\, \mathsf{P}(S_{i - 1} = k + 1) \\
    &\qquad \qquad \qquad \qquad + \mathsf{P}(S_i \leq k \mid S_{i - 1} > k + 1)\, \mathsf{P}(S_{i - 1} > k + 1) \\
    &= p_f\, \mathsf{P}(S_{i - 1} = k + 1) + 0\, \mathsf{P}(S_{i - 1} > k + 1) \\
    &\leq p_f\, \leq \mathsf{P}(S_i \leq k \mid S_{i - 1} \leq k).
\end{align*}

As a result

\begin{align}
    \mathsf{P}(S_i \leq k) &= \mathsf{P}(S_i \leq k \mid S_{i - 1} \leq k) \,\mathsf{P}(S_{i - 1} \leq k) + \mathsf{P}(S_i \leq k \mid S_{i - 1} > k) \,\mathsf{P}(S_{i - 1} > k) \nonumber \\
    &\leq \mathsf{P}(S_i \leq k | S_{i - 1} \leq k). \label{eq:conditional-bigger}
\end{align}

The probability of a random walk of length $n$ to stay between the first $k \geq 0$ positions is thus
\begin{align*}
    \mathsf{P}(S_i \leq k, \, \forall{i \leq n}) &= \mathsf{P}(S_0 \leq k) \prod_{i = 1}^n \mathsf{P}(S_i \leq k \mid S_{j} \leq k, \, \forall j \in [0, i - 1]) \\
    &= \prod_{i = 1}^n \mathsf{P}(S_i \leq k \mid S_{i - 1} \leq k), \text{ with } \mathsf{P}(S_0 \leq k) = 1, \, \forall k \geq 0 \\
    &\geq  \prod_{i = 1}^n \mathsf{P}(S_i \leq k), \text{ by Eq. }\eqref{eq:conditional-bigger}\\
    &= \prod_{j=1}^n \left(\sum_{i=0}^k p_{i, j}\right) \\
    &\geq \left(\sum_{i=0}^k \pi_i \right)^n,
\end{align*}
where $\pi_i$ denotes the stationary probability of the random walk for $i \in \mathbb{N}$. From the recursive relation, we get $\pi_i = (\frac{p_f}{1 - p_f}) \pi_{i + 1}$, which means $\pi_i = (\frac{1 - p_f}{p_f})^i \pi_0$. 

We now calculate the probability $\pi_{\leq k}$ of a randomly chosen point to be part of the first $k$ positions for the stationary distribution
\begin{align*}
    \pi_{\leq k} = \sum_{i=0}^k \pi_i  = \frac{\pi_0 \sum_{i=0}^k (\frac{1 - p_f}{p_f})^i}{\pi_0 \sum_{i=0}^{\infty} (\frac{1 - p_f}{p_f})^i} = 1 - \left(\frac{1 - p_f}{p_f}\right)^{k + 1}.
\end{align*}
Thus the required probability must be lower bounded by
\begin{align}
    \pi_{\leq k}^n \geq \delta \implies \left(1 - \left(\frac{1 - p_f}{p_f}\right)^{k + 1}\right)^n &\geq \delta \nonumber \\
    \log\left(1 - \left(\frac{1 - p_f}{p_f}\right)^{k + 1}\right) &\geq \frac{1}{n} \log(\delta) \nonumber \\
    \left(\frac{1 - p_f}{p_f}\right)^{k + 1} &\leq 1 - e^{\frac{1}{n}\log(\delta)} \nonumber \\
    k &\geq \frac{\log\left(1 - e^{\frac{1}{n}\log(\delta)}\right)}{\log\left(\frac{1 - p_f}{p_f}\right)} - 1,
\end{align}
where for the last step we used the fact that $\log \left(\frac{1 - p_f}{p_f} \right) < 0$ since $p_f > \frac12$ implies $\frac{1 - p_f}{p_f} < 1$.
\end{proof}

We can now move to the proof of Theorem~\ref{theorem:min-max-convergence}.
\begin{proof}
We denote with $c_x$ the constant used for the sufficient decrease condition of the min problem. We first prove the deterministic case. In the deterministic case, Theorems~\ref{theorem:convergence-nonconvex} and \ref{theorem:convergence-pl} hold deterministically, meaning that we can reduce the norm of the gradient below a threshold $\epsilon$ for the nonconvex case in $\mathcal{O}(\epsilon^{-2})$ iterations and for the case that the function satisfies our PL condition in $\mathcal{O}(\log(\epsilon^{-1}))$ iterations.

At each step, the max problem is solved almost exactly, which is guaranteed by Theorem~\ref{theorem:convergence-pl} and Algorithm~\ref{algo:direct-search}. Then 

\begin{equation*}
    \| \nabla_{\y} f(\x_{t-1}, \y_{t}) \| \leq \epsilon^{\max},
\end{equation*}

for an accuracy $\epsilon^{\max}$ to be specified later. In the proof, we will show that for a particular choice of a forcing function constant, the improvement on the minimization problem is better than possible deterioration caused by the updates of the max problem. By Assumption~\ref{ass:lipschitz} of Lipschitz continuity 

\begin{align}
    \| \nabla_{\y} f(\x_t, \y_t) - \nabla_{\y} f(\x_{t-1}, \y_t) \| &\leq L_{12} \| \x_t - \x_{t-1} \| = L_{12} \sigma_t \nonumber \\
    \implies \| \nabla_{\y} f(\x_t, \y_t)\| &\leq L_{12} \sigma_t + \epsilon^{\max}, \label{eq:nabla_y}
\end{align}
for a successful update. Here $\sigma_t$ is used to denote the step size used for the minimization step throughout Algorithms~\ref{algo:min-max-DR} and~\ref{algo:one-step-DR}. We note that $\sigma_t$ always belongs to a successful step, by the notation used in Algorithm~\ref{algo:min-max-DR}. Also by triangle inequality we have that (let $\y_{t}^*$ and $\y_{t+1}^*$ belong to the optimal solution sets at iterations $t$ and $t+1$ respectively)

\begin{align}
    \| \y_{t+1} - \y_t \| &= \| \y_{t+1} - \y_{t+1}^* + \y_{t+1}^* - \y_t^* + \y_t^* - \y_t \|  \nonumber \\ 
    &\leq \| \y_{t+1} - \y_{t+1}^* \| + \| \y_{t+1}^* - \y_t^* \| + \| \y_t^* - \y_t \|. 
\end{align}

By Lemma~\ref{lemma:min-max-connection} we have that $\| \y_{t+1}^* - \y_t^* \| \leq L_{xy} \sigma_t$, since $\y_{t+1}^* \in \mathbbm{Y}(\x_{t})$ and $\y_{t}^* \in \mathbbm{Y}(\x_{t-1})$ (we remind that $\mathbbm{Y}(\x) = \argmax_{\y}f(\x, \y)$). Also, as a consequence of Definition~\ref{def:PL-condition} and Lemma~\ref{lemma:pl-epsilon} we have that both

\begin{equation*}
    \| \y_{t+1} - \y_{t+1}^* \| \leq \frac{\epsilon^{\max}}{2 \mu} \quad \text{ and } \quad \| \y_{t} - \y_{t}^* \| \leq \frac{\epsilon^{\max}}{2 \mu}.
\end{equation*}

As a result

\begin{align}
    \| \y_{t+1} - \y_t \| \leq \frac{\epsilon^{\max}}{\mu} + L_{xy} \sigma_t. \label{eq:y_distance}
\end{align}

Finally, for a successful update of the Algorithm~\ref{algo:one-step-DR} we have

\begin{align}
     f(\x_{t}, \y_{t + 1}) - f(\x_{t}, \y_{t}) &\leq \langle \nabla_{\y} f(\x_{t}, \y_{t}), \y_{t+1} - \x_t \rangle + \frac{L_{22}}{2} \| \x_{t+1} - \x_t \|^2 \nonumber \\
     &\leq (L_{12} \sigma_t + \epsilon^{\max}) (L_{xy}\sigma_t + \frac{\epsilon^{\max}}{\mu}) \nonumber \\ 
     &\quad + \frac{L_{22}}{2} (L_{xy}\sigma_t + \frac{\epsilon^{\max}}{\mu})^2 \nonumber \\
     &= D_1 \sigma_t^2 + D_2 \sigma_t \epsilon^{\max} + D_3 (\epsilon^{\max})^2,
\end{align}

for $D_1 \triangleq L_{12}L_{xy} + \frac{L_{22}}{2} L_{xy}^2 $, $D_2 = \frac{L_{12}}{\mu} + L_{xy} + \frac{L_{22}L_{xy}}{\mu}$ and $D_3 = \frac{1}{\mu}(1 + \frac{L_{22}}{2\mu})$. During the updates of the minimization problem we have that
\begin{equation*}
    \sigma \geq \sigma_{\min} = \sigma_{\epsilon}, \, \text{ with } \sigma_{\epsilon} = C \epsilon.
\end{equation*}
Here $C$, which is defined in Lemma~\ref{lemma:unsuccessful-step}, entails the constants for the min problem. We want to ensure that

\begin{equation}
    f(\x_{t}, \y_{t + 1}) - f(\x_{t - 1}, \y_{t}) < - K \sigma_t^2,
\end{equation}

for some $K > 0$ and for $\sigma_t \geq \sigma_{\min}$, to then apply Theorem~\ref{theorem:pfi-k-nonconvex}, for $f^*$ the minimum of $f$ at each $\y_t$. Taking also into account our sufficient decrease condition, we want to make sure that the following holds for the polynomial $p$

\begin{equation}
    p(\sigma_t) \triangleq K \sigma_t^2 - c_x \sigma_t^2 + D_1 \sigma_t^2 + D_2 \sigma_t \epsilon^{\max} + D_3 (\epsilon^{\max})^2 \leq 0
\end{equation}

for every $\sigma_t \geq \sigma_{\min}$. To establish this we just need to ensure that for the quadratic with negative second degree coefficient (for $c_x > D_1 + K$) the maximum occurs at position:

\begin{equation}
    \frac{D_2 \epsilon^{\max}}{2(c_x - K - D_1)} \leq C \epsilon \implies \epsilon^{\max} \leq \epsilon \frac{2C(c_x - K - D_1)}{D_2} \label{eq:e-max-con1}
\end{equation}

and also that

\begin{align}
    p(C \epsilon) &\leq 0 \iff \nonumber \\
    (- c_x + K + D_1)C^2\epsilon^2 + D_2C \epsilon \epsilon^{\max} + D_3(\epsilon^{\max})^2 &\leq 0.
\end{align}

For the final condition to hold

\begin{equation}
\epsilon^{\max} \leq \epsilon \frac{C (- D_2 + \sqrt{D_2^2 + 4(c_x - K -D_1) D_3}}{2D_3}.
\end{equation}

In the stochastic case, we apply Theorems~\ref{theorem:convergence-nonconvex} and \ref{theorem:convergence-pl} as is to get the expected number of steps. In this case however, the step size may become smaller than the pre-specified $\sigma_{\min}$ parameter, due to inaccurate estimates. We can then use Lemma~\ref{lemma:random-walk}, to get a bound with high probability, regarding this minimum step size value. More specifically for the number of iterates $n$ specified by Theorem~\ref{theorem:convergence-nonconvex}, for $k \geq \frac{\log(1 - e^{\frac{1}{n}\log(\delta)})}{\log(\frac{1 - p_x}{p_x})} - 1$, throughout the updates
\begin{equation*}
    \sigma \geq \sigma_{\min}^{\prime} = \frac{1}{\gamma^k} \sigma_{\min},
\end{equation*}
with probability at least $\delta > 0$, where $\gamma$ is the update parameter for the min problem in Algorithm~\ref{algo:one-step-DR}. We then get the similar bounds

\begin{equation*}
    \epsilon^{\max} \leq \epsilon \,\min\left\{ \frac{2C(c_x - K - D_1)}{\gamma^k D_2}, \; \frac{C (- D_2 + \sqrt{D_2^2 + 4(c_x - K -D_1) D_3}}{2D_3\gamma^k}\right\}.
\end{equation*}

We note that $K$ acts as a new sufficient decrease constant and should be taken into account for all assumptions of Theorem~\ref{theorem:convergence-nonconvex}, namely $K > 2 \epsilon_x$, which holds for the constant $c_x > D_1 + K > D_1 + 2 \epsilon_x$.

\end{proof}


\section{Experimental Setup} \label{app:exp}

\subsection{Robust Optimization}

The Wisconsin breast cancer data set, is a binary classification task with 569 samples in total, each having 30 attributes. We use a simple neural network with a hidden layer of size 50 and a LeakyReLU activation. This choice of activation accommodates the GDA baseline providing additional gradient information. All networks across methods and folds are initialized with the same weights. For the GDA method we tried a range of different learning rates from the set \{0.1, 0.05, 0.01, 0.005, 0.001, 0.0005\}, but only present results for the cases that converged. 

In Fig.~\ref{fig:robust-speed-of-convergence} we present the evolution of the zero-one error across epochs for each method. We stress that one epoch for the GDA approach corresponds to one update each for the max and the min problem, whereas one epoch for DR corresponds to a series of updates for the max problem (at most 10) followed by a single update for the min problem. GDA was run for a total of 10000 epochs and DR for a maximum of 2000 epochs but usually converges a lot faster than that. GDA suffers considerably more by poor initializations compared to DR. In Fig.~\ref{fig:robust-speed-of-convergence} constant large errors correspond to a constant output of the network for a specific class of the problem (for this unbalanced dataset with rates 0.63 and 0.37).

\begin{figure}[!ht]
    \centering
    \includegraphics[width=1\textwidth]{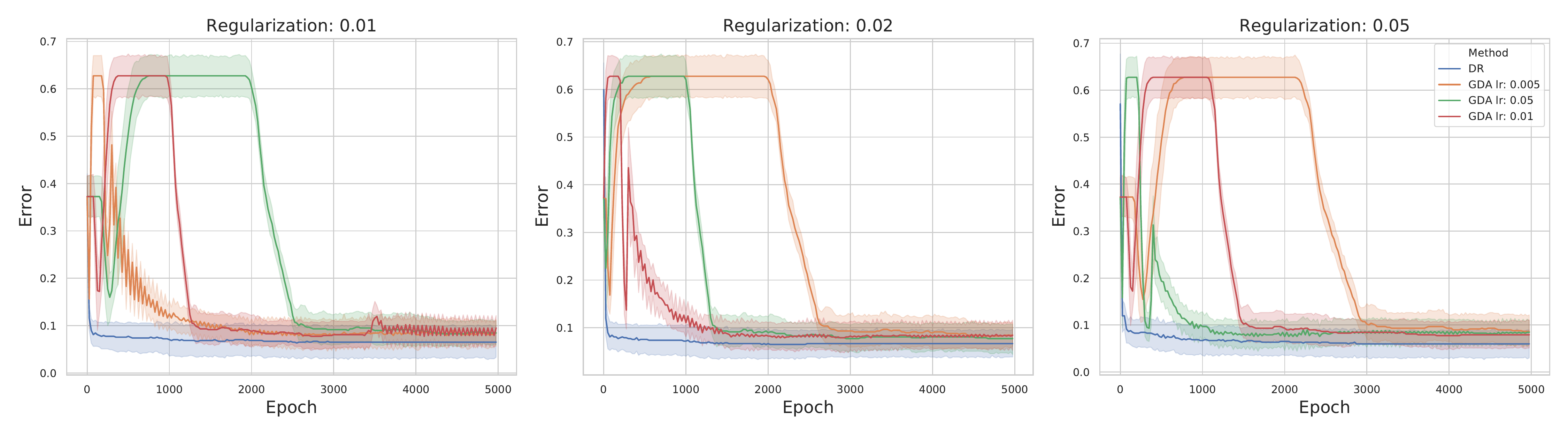}
    \caption{Misclassification error across epochs for each method.}
    \label{fig:robust-speed-of-convergence}
\end{figure}


\subsection{Toy Examples}

Although in the examples following, the objective of the max player is nonconcave and does not satisfy the PL condition, empirical results demonstrate that the proposed algorithm can be successful. We begin by illustrating examples of GANs learning different 2D underlying distributions for a continuous case in Fig.~\ref{fig:mixtures-continuous-case}. Both the generator and the discriminator have 2 hidden layers of size 20 (64 for learning a mixture of Gaussian in a grid formation) with Tanh activations, while we also use spectral normalization for the discriminator. In all scenarios, we sample the latent code from a lower-dimensional space $N(0, I_2)$, such that it matches the data dimensionality, allowing the generator to learn a simpler mapping (as in~\cite{grnarova2017online}).

\begin{figure}[!ht]
    \centering
    \includegraphics[width=1\textwidth]{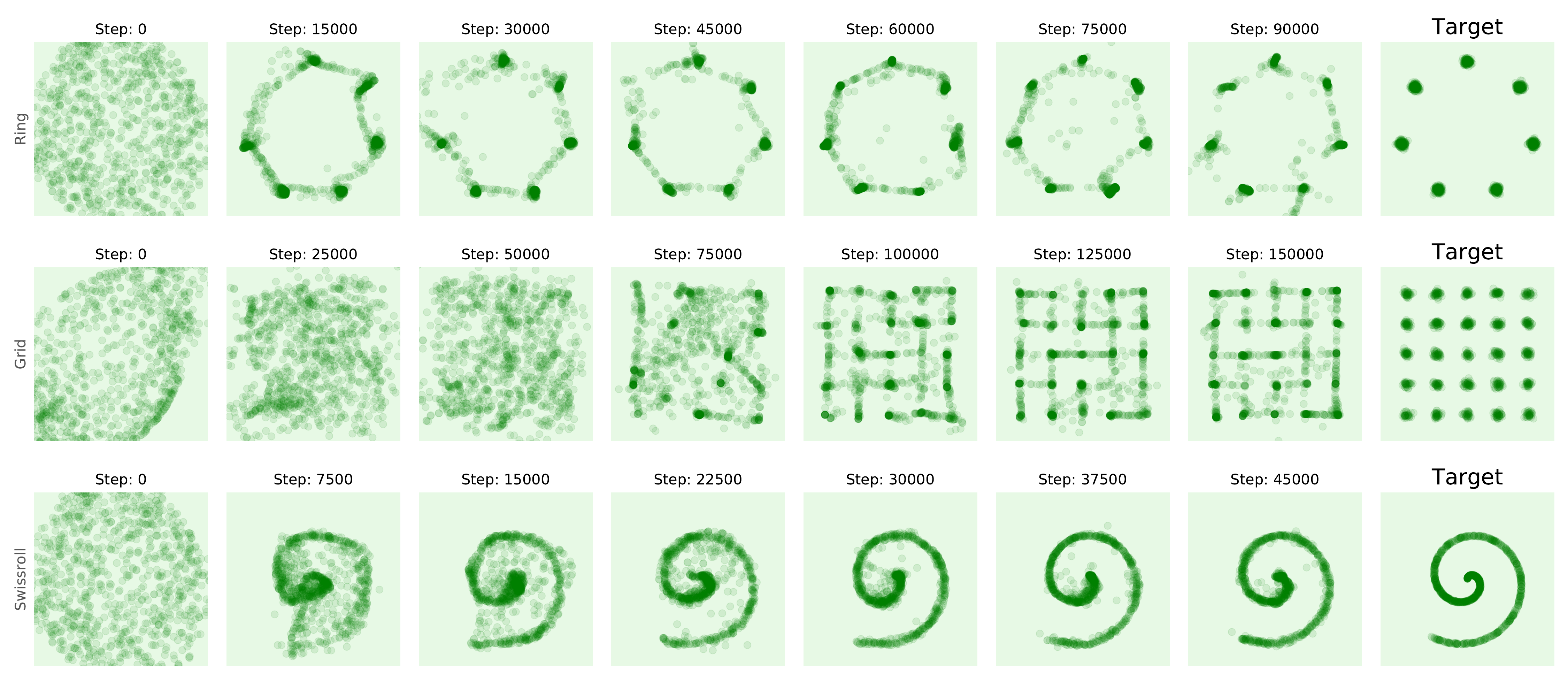}
    \caption{Mode collapse check for direct-search methods for three different problems in the continuous setting.}
    \label{fig:mixtures-continuous-case}
\end{figure}

Motivated by encouraging results, we proceed in a discrete setting, where each of the 2 dimensions of the underlying distributions is parametrized by a categorical variable. The choice of this categorical variable makes the objective function of the generator nondifferentiable. As aforementioned, our algorithm can support multi-categorical data. In the current literature, the most popular methods to deal with this kind of scenario are baselines based on the Gumbel-softmax or the REINFORCE algorithm. Due to their sampling techniques though, dependence on the number of parameters of the model is exponential for these baselines. 

We describe shortly how training is performed for each of the baselines used. Based on the output logits $o$ of size $n$, the result of a projection layer, each method samples a new point $y$.

\paragraph{Gumbel-softmax}
Using the Gumbel-max trick, the sampling can be parametrized as

\begin{equation}
    y = \text{one-hot}(\argmax_{1 \leq i \leq n}(o^{(i)} + g^{(i)})),
\end{equation}

where $g$ are sampled from the \textit{i.i.d.} Gumbel distribution. To enable the calculation of gradients, this is relaxed to the result of a softmax operation

\begin{equation}
    \hat{y} = \sigma(\frac{o + g}{T}),
\end{equation}

with $T > 0$, the temperature, controlling the softness of the sampling. A high initial temperature value forces more exploration. As the temperature decreases, $\hat{y}$ becomes a better approximation of $y$, leading however to steeper gradients and more instabilities. This is also the reason why gradient clipping is crucial for the stability of this method. Untimely updates of the temperature have been known to bolster mode collapse. For all experiments, we use an exponential decay update scheme, decreasing the temperature after a predefined number of steps.

\paragraph{REINFORCE} We sample a new point $s$ from the output logits and appoint a specific reward according to the output of the discriminator $r = 2*(\D(s) - 0.5)$ (we remind that the discriminator uses a sigmoid output), rewarding positively samples that manage to fool the discriminator and negatively those that fail to do so. Subtracting the baseline value of $0.5$ helps reduce variance. To alleviate the large variance introduced by the sampling, we increase the number of steps taken by the generator compared to the other methods.

\paragraph{Direct-search} Direct-search method just chooses the output with the highest probability

\begin{equation}
    y = \text{one-hot}(\argmax_{1 \leq i \leq n}o^{(i)}).
\end{equation}

For the discrete toy example illustrated, we draw samples from an evenly weighted and evenly spaced mixture of Gaussians with 7 components in a 2-dimensional space. The 2d sample is then discretized, according to a specific level, leading to 51 possible values for each of the two dimensions of the problem. These categorical data-points are then transformed to the corresponding continuous one, by a linear mapping, and given as input to the discriminator. This leads to an ordinal relationship between the categorical points (we only display ordinal data for visualization purposes). To monitor the learning curve of the generator compared to the true distribution, we also calculate the Hellinger distance, which is a nonparametric method that calculates the difference of two discrete distributions as
\begin{equation}
    \mathcal{H} (\mathcal{P}, \mathcal{Q}) = \frac{1}{\sqrt{2}} \sqrt{\sum_{i=1}^k(\sqrt{p_i} - \sqrt{q_i})^2},
\end{equation}
\vspace{-1mm}

as well as the maximum mean discrepancy, comparing the generated samples with the underlying data. Note that this example suffers from stochasticity due to the way samples are collected and batches are created. Results for all baselines are provided in Figure~\ref{fig:mixtures-discrete-case}.

\begin{figure}[!ht]
    \centering
    \includegraphics[width=1\textwidth]{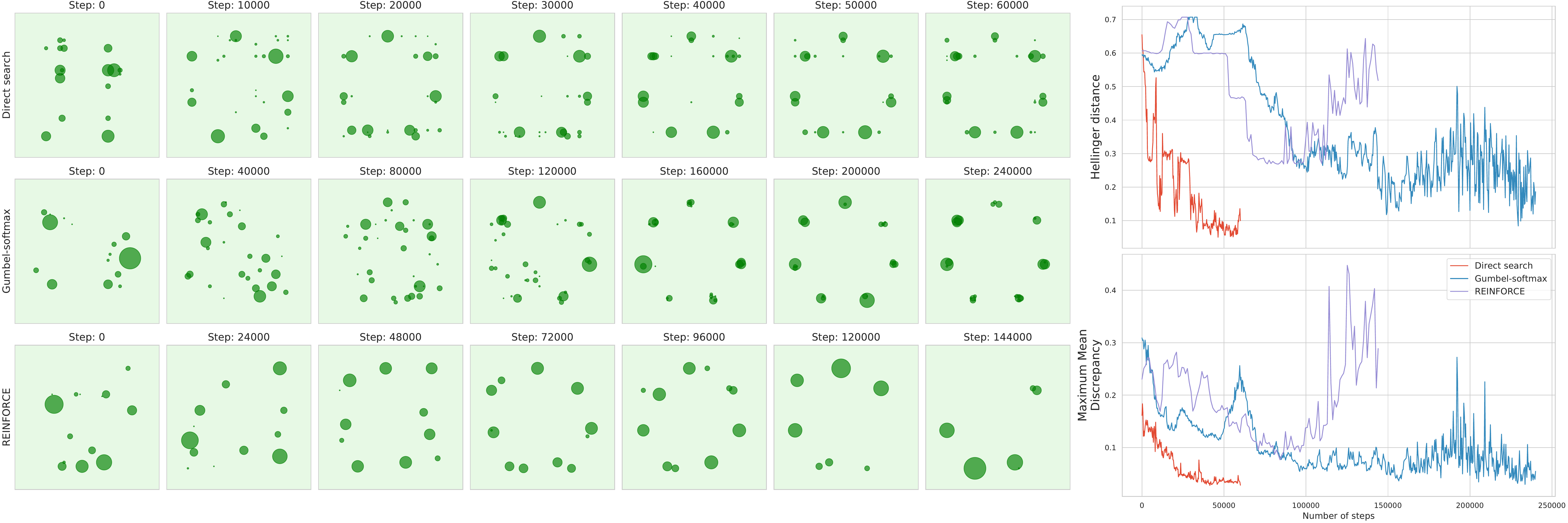}
    \caption{Comparison of direct-search with baselines for the discrete toy example.}
    \label{fig:mixtures-discrete-case}
\end{figure}

When using DR, to overcome the nonsmoothness of the objective function, increasing the $\ell_2$ regularization can help with the convergence. Too large of an increase inevitably leads to mode collapse. On the other hand, too small of a regularization parameter requires a large enough step size parameter to enable progress and thus can make learning more difficult. In general, we found that DR converges to similar solutions for a wide range of regularization parameter choices without significant variation in terms of the number of steps required to do so.
\end{appendices}

\end{document}